\documentclass[11pt,reqno]{amsart}

\usepackage{amsmath}
\usepackage{amssymb}
\usepackage{amsthm}

\newtheorem{theorem}{Theorem}[section]
\newtheorem{prop}[theorem]{Proposition}
\newtheorem{lem}[theorem]{Lemma}

\theoremstyle{definition}
\newtheorem{defn}[theorem]{Definition}
\theoremstyle{remark}
\newtheorem*{rem}{Remark}
\numberwithin{equation}{section}

\begin{document}

\title[Quantum double and the Weyl algebra]
{Twisting of the Quantum double and the Weyl algebra}

\author{Byung-Jay Kahng}
\date{}
\address{Department of Mathematics and Statistics\\ Canisius College\\
Buffalo, NY 14208}
\email{kahngb@canisius.edu}

\begin{abstract}
Quantum double construction, originally due to Drinfeld and has been since generalized 
even to the operator algebra framework, is naturally associated with a certain 
(quasitriangular) $R$-matrix ${\mathcal R}$.  It turns out that ${\mathcal R}$ 
determines a twisting of the comultiplication on the quantum double.  It then suggests 
a twisting of the algebra structure on the dual of the quantum double.  For $D(G)$, 
the $C^*$-algebraic quantum double of an ordinary group $G$, the ``twisted $\widehat{D(G)}$'' 
turns out to be the Weyl algebra $C_0(G)\times_{\tau}G$, which is in turn isomorphic to 
${\mathcal K}(L^2(G))$.  This is the $C^*$-algebraic counterpart to an earlier 
(finite-dimensional) result by Lu.  It is not so easy technically to extend this program 
to the general locally compact quantum group case, but we propose here some possible 
approaches, using the notion of the (generalized) Fourier transform.
\end{abstract}
\maketitle

\section{Introduction}

There are a few different approaches to formulate the notion of quantum groups, 
which are generalizations of ordinary groups.  In the finite-dimensional case, they
usually come down to Hopf algebras \cite{Ab}, \cite{Mo}, although there actually 
exist examples of quantum groups that cannot be described only by Hopf algebra 
languages.  More generally, the approaches to quantum groups include the (purely 
algebraic) setting of ``quantized universal enveloping (QUE) algebras'' \cite{Dr}, 
\cite{CP}; the setting of multiplier Hopf algebras and algebraic quantum groups 
\cite{VDmult}, \cite{KuVD}; and the ($C^*$- or von Neumann algebraic) setting of 
locally compact quantum groups \cite{KuVa}, \cite{KuVavN}, \cite{MNW}, \cite{VDvN}. 
In this paper, we are mostly concerned with the setting of $C^*$-algebraic locally 
compact quantum groups.

In all these approaches to quantum groups, one important aspect is that the category
of quantum groups is a ``self-dual'' category, which is not the case for the (smaller)
category of ordinary groups.  To be more specific, a typical quantum group $A$ is
associated with a certain dual object $\hat{A}$, which is also a quantum group, and
the dual object, $\hat{\hat{A}}$, of the dual quantum group is actually isomorphic
to $A$.  This result, $\hat{\hat{A}}\cong A$, is a generalization of the Pontryagin
duality, which holds in the smaller category of abelian locally compact groups.

For a finite dimensional Hopf algebra $H$, its dual object is none other than the 
dual vector space $H'$, with its Hopf algebra structure obtained naturally from 
that of $H$.  In general, however, a typical quantum group $A$ would be infinite 
dimensional, and in that case, the dual vector space is too big to be given any 
reasonable structure (For instance, one of the many drawbacks is that $(A\otimes A)'$ 
is strictly larger than $A'\otimes A'$.).  

In each of the approaches to quantum groups, therefore, a careful attention should 
be given to making sense of what the dual object is for a quantum group, as well as 
to exploring the relationship between them.  This is especially true for the 
analytical settings, where the quantum groups are required to have additional, 
topological structure.  The success of the locally compact quantum group framework 
by Kustermans and Vaes \cite{KuVa}, and also by Masuda, Nakagami, and Woronowicz 
\cite{MNW} is that they achieve the definition of locally compact quantum groups 
so that it has the self-dual property.

Meanwhile, given a Hopf algebra $H$ and its dual $\hat{H}$, there exists the notion 
of the ``quantum double'' $H_D=\hat{H}^{\operatorname{op}}\Join H$ (see \cite{Dr}, 
\cite{Mo}).  This notion can be generalized even to the setting of locally compact 
quantum groups: From a von Neumann algebraic quantum group $(N,\Delta)$, one can 
construct the quantum double $(N_D,\Delta_D)$.  See Section~2 below.

The quantum double is associated with a certain ``quantum universal $R$-matrix'' 
type operator ${\mathcal R}\in N_D\otimes N_D$.  It turns out that ${\mathcal R}$ 
determines a left cocycle for $\Delta_D$, and allows us to {\em twist (or deform)\/} 
the comultiplication on $N_D$, or its $C^*$-algebraic counterpart $A_D$.  The result, 
$(A_D,{\mathcal R}\Delta_D)$, can no longer become a locally compact quantum group, 
but it suggests a twisting of the algebra structure at the level of $\widehat{A_D}$, 
the dual of the quantum double.  Our intention here is to explore this algebra, the 
``deformed $\widehat{A_D}$''.

There are two crucial obstacles in carrying out this program.  For one thing, 
the $C^*$-algebra $\widehat{A_D}$ itself can be rather complicated in general. 
In addition, unlike in the algebraic approaches, even the simple tool like the 
dual pairing is not quite easy to work with.  In the locally compact quantum 
group framework, the dual pairing between a quantum group $A$ and its dual 
$\hat{A}$ is defined at dense subalgebra level, by using the multiplicative 
unitary operator associated with $A$ and $\hat{A}$.  While it is a correct 
definition (in the sense that it is a natural generalization of the obvious 
dual pairing between $H$ and $H'$ in the finite-dimensional case), the way 
it is defined makes it rather difficult to work with.  For instance, there is 
no straightforward way of obtaining a dual object of a $C^*$-bialgebra.

These technical difficulties cannot be totally overcome, but we can improve the 
situation by having a better understanding of the duality picture.  Recently 
in \cite{BJKqfourier}, motivated by Van Daele's work in the multiplier Hopf 
algebra framework \cite{VDFourier}, the author defined the (generalized) Fourier 
transform between a locally compact quantum group and its dual.  In addition, 
an alternative description of the dual pairing was found (see Section~4 of 
\cite{BJKqfourier}), in terms of the Haar weights and the Fourier transform.
This alternative perspective to the dual pairing is useful in our paper. 

In the case of an ordinary locally compact group $G$, so for $A=C_{red}^*(G)$ 
(the ``reduced group $C^*$-algebra'') and $\hat{A}=C_0(G)$, the quantum double 
turns out to be $A_D=C_0(G)\rtimes_{\alpha}G$, the crossed product $C^*$-algebra 
given by the group $G$ acting on itself by conjugation $\alpha$.  It is also 
known that $\widehat{A_D}=C_{red}^*(G)\otimes C_0(G)$.  After carrying out 
the twisting process of $\widehat{A_D}$ as described above, we will see in 
Section~5 below that it gives rise to the crossed product $C^*$-algebra $B=C_0(G)
\rtimes_{\tau}G$, where $\tau$ is the translation.  This algebra is often called 
the ``Weyl algebra''.  It is quite interesting to observe this relationship 
between the quantum double (a quantum group) and the Weyl algebra (no longer 
a quantum group), which are both well-known to appear in some physics applications.

Meanwhile, it is known that as a $C^*$-algebra, the Weyl algebra is isomorphic 
to the algebra of compact operators: $C_0(G)\rtimes_{\tau}G\cong{\mathcal K}
\bigl(L^2(G)\bigr)$.  In the (finite-dimensional) Hopf algebra setting, a similar 
process was carried out by Lu \cite{Ludouble}, \cite{Mo}: Lu's result says that 
the twisting of the dual of the quantum double turns out to be isomorphic to 
the smash product $H\#\hat{H}$, which is in turn known to be isomorphic to 
$\operatorname{End}(H)$.  In this sense, our observation here will be the 
$C^*$-algebraic counterpart to Lu's result.  See also, \cite{DeVD}, where the 
result is obtained in the setting of multiplier Hopf algebras.

Motivated by the results in these ``good'' cases, we then try to consider the case 
of general locally compact quantum groups.  While there are technical obstacles, 
we propose in Section~6 a workable approach based on the property of the Fourier 
transform.  For a general (not necessarily regular) locally compact quantum group 
$A$, the $C^*$-algebra of ``deformed $\widehat{A_D}$'' may no longer be isomorphic 
to ${\mathcal K}({\mathcal H})$ and can be quite complicated: It may not even be 
of type I.

Here is how the paper is prepared: In Section~2, we give basic definitions and 
review some results about locally compact quantum groups and its dual.  We will 
also describe the dual pairing map, including an alternative characterization 
obtained recently by the author.

In Section~3, we will discuss the quantum double construction.  This is a special 
case of the ``double crossed product'' construction developed by Baaj and Vaes 
in \cite{BjVa}.  However, the scope of that paper is a little too general, 
and we needed to have an explicit summary written out on the quantum double 
construction for a general locally compact quantum group.  Some of the results here, 
while straightforward, were just barely noted in \cite{BjVa} and have not appeared 
elsewhere: Among such results is the discussion on the ``quantum $R$-matrix'' type 
operator.  In Section~4, we will see how the $R$-matrix ${\mathcal R}$ determines 
a left twisting of the comultiplication on the quantum double.  It will suggest 
a twisting (deformation) at the dual level.

In Section~5, we consider the case of an ordinary group and its quantum double 
$D(G)$, then carry out the twisting of $\widehat{D(G)}$.  As noted above, the 
result is shown to be isomorphic to the Weyl algebra.  In Section~6, we consider 
the general case.  Using the case of $D(G)$ and $\widehat{D(G)}$ as a basis, we 
will collect some information that can be used in our efforts to go further into 
the case of general locally compact quantum groups.  We will propose here a reasonable 
description for the deformed $\widehat{A_D}$.  The notion of the generalized Fourier 
transform defined in \cite{BJKqfourier} will play a central role.

\section{Preliminaries}

\subsection{Locally compact quantum groups}
Let us first begin with the definition of a {\em von Neumann algebraic locally compact 
quantum group\/}, given by Kustermans and Vaes \cite{KuVavN}.  This definition is 
known to be equivalent to the definition in the $C^*$-algebra setting \cite{KuVa}, 
and also to the formulation given by Masuda--Nakagami--Woronowicz \cite{MNW}.  
Refer also to the recent paper by Van Daele \cite{VDvN}.  We note that the existence 
of the Haar (invariant) weights has to be assumed as a part of the definition.

\begin{defn}
Let $M$ be a von Neumann algebra, together with a unital normal 
${}^*$-homomorphism $\Delta:M\to M\otimes M$ satisfying the ``coassociativity'' 
condition: $(\Delta\otimes\operatorname{id})\Delta=(\operatorname{id}
\otimes\Delta)\Delta$.  Assume further the existence of a left invariant 
weight and a right invariant weight, as follows:
\begin{itemize}
\item $\varphi$ is an n.s.f. weight on $M$ that is left invariant:
$$
\varphi\bigl((\omega\otimes\operatorname{id})(\Delta x)\bigr)=\omega(1)\varphi(x),
\quad {\text {for all $x\in{\mathfrak M}_{\varphi}^+$ and $\omega\in M^+_*$.}}
$$
\item $\psi$ is an n.s.f. weight on $M$ that is right invariant:
$$
\psi\bigl((\operatorname{id}\otimes\omega)(\Delta x)\bigr)=\omega(1)\psi(x),
\quad {\text {for all $x\in{\mathfrak M}_{\psi}^+$ and $\omega\in M^+_*$.}}
$$
\end{itemize}
Then we say that $(M,\Delta)$ is a {\em von Neumann algebraic quantum group\/}.
\end{defn}

\begin{rem}
We are using the standard notations and terminologies from the theory of weights.
For instance, an ``n.s.f. weight'' is a normal, semi-finite, faithful weight.
For an n.s.f. weight $\varphi$, we write $x\in{\mathfrak M}_{\varphi}^+$ to mean 
$x\in M^+$ so that $\varphi(x)<\infty$, while $x\in{\mathfrak N}_{\varphi}$ means 
$x\in M$ so that $\varphi(x^*x)<\infty$.  See \cite{Tk2}.  Meanwhile, it can be 
shown that the Haar weights $\varphi$ and $\psi$ above are unique, up to scalar 
multiplication.
\end{rem}

Let us fix $\varphi$.  Then by means of the GNS construction $({\mathcal H},\iota,
\Lambda)$ for $\varphi$, we may as well regard $M$ as a subalgebra of the operator 
algebra ${\mathcal B}({\mathcal H})$, such as $M=\iota(M)\subseteq{\mathcal B}
({\mathcal H})$.  Thus we will have: $\bigl\langle\Lambda(x),\Lambda(y)\bigr\rangle
=\varphi(y^*x)$ for $x,y\in{\mathfrak N}_{\varphi}$, and $a\Lambda(y)=\iota(a)
\Lambda(y)=\Lambda(ay)$ for $y\in{\mathfrak N}_{\varphi}$, $a\in M$.  Consider next 
the operator $T$, which is the closure of the map $\Lambda(x)\mapsto\Lambda(x^*)$ for 
$x\in{\mathfrak N}_{\varphi}\cap{\mathfrak N}_{\varphi}^*$.  Expressing its polar 
decomposition as $T=J\nabla^{1/2}$, we obtain in this way the ``modular operator'' 
$\nabla$ and the ``modular conjugation'' $J$.  The operator $\nabla$ determines 
the modular automorphism group.  Refer to the standard weight theory \cite{Tk2}.

Meanwhile, there exists a unitary operator $W\in{\mathcal B}({\mathcal H}\otimes
{\mathcal H})$, called the {\em multiplicative unitary operator\/} for $(M,\Delta)$. 
It is defined by $W^*\bigl(\Lambda(x)\otimes\Lambda(y)\bigr)=(\Lambda\otimes\Lambda)
\bigl((\Delta y)(x\otimes1)\bigr)$, for $x,y\in{\mathfrak N}_{\varphi}$.  It satisfies 
the pentagon equation of Baaj and Skandalis \cite{BS}: $W_{12}W_{13}W_{23}
=W_{23}W_{12}$.  We also have: $\Delta a=W^*(1\otimes a)W$, for $a\in M$.  The 
operator $W$ is the ``left regular representation'', and it provides the following 
useful characterization of $M$:
$$
M=\overline{\{(\operatorname{id}\otimes\omega)(W):\omega\in{\mathcal B}
({\mathcal H})_*\}}^w\,\bigl(\subseteq{\mathcal B}({\mathcal H})\bigr),
$$
where $-^w$ denotes the von Neumann algebra closure (for instance, the closure 
under $\sigma$-weak topology).

If we wish to consider the quantum group in the $C^*$-algebra setting, we just 
need to take the norm completion instead, and restrict $\Delta$ to $A$.  See 
\cite{KuVa}, \cite{VDvN}.  Namely,
$$
A=\overline{\{(\operatorname{id}\otimes\omega)(W):\omega\in{\mathcal B}
({\mathcal H})_*\}}^{\|\ \|}\,\bigl(\subseteq{\mathcal B}({\mathcal H})\bigr).
$$

Constructing the antipode is rather technical (it uses the right Haar weight), 
and we refer the reader to the main papers \cite{KuVa}, \cite{KuVavN}.  See 
also an improved treatment given in \cite{VDvN}, where the antipode is defined
in a more natural way by means of Tomita--Takesaki theory.  For our purposes, 
we will just mention the following useful characterization of the antipode $S$:
\begin{equation}\label{(antipode)}
S\bigl((\operatorname{id}\otimes\omega)(W)\bigr)
=(\operatorname{id}\otimes\omega)(W^*).
\end{equation}
In fact, the subspace consisting of the elements $(\operatorname{id}\otimes\omega)
(W)$, for $\omega\in{\mathcal B}({\mathcal H})_*$, is dense in $M$ and forms 
a core for $S$.  Meanwhile, there exist a unique ${}^*$-antiautomorphism 
$R$ (called the ``unitary antipode'') and a unique continuous one parameter 
group $\tau$ on $M$ (called the ``scaling group'') such that we have: $S=
R\tau_{-\frac{i}{2}}$.  Since $(R\otimes R)\Delta=\Delta^{\operatorname{cop}}R$, 
where $\Delta^{\operatorname{cop}}$ is the co-opposite comultiplication 
(i.\,e. $\Delta^{\operatorname{cop}}=\chi\circ\Delta$, for $\chi$ the flip map on 
$M\otimes M$), the weight $\varphi\circ R$ is right invariant.  So we can, without 
loss of generality, choose $\psi$ to equal $\varphi\circ R$.  The GNS map for 
$\psi$ will be written as $\Gamma$.

From the right Haar weight $\psi$, we can find another multiplicative 
unitary $V$, defined by $V\bigl(\Gamma(x)\otimes\Gamma(y)\bigr)=(\Gamma
\otimes\Gamma)\bigl(\Delta x)(1\otimes y)\bigr)$, for $x,y\in{\mathfrak N}_{\psi}$. 
It is the ``right regular representation'', and it provides an alternative 
characterization of $M$:  That is, $M=\overline{\{(\omega\otimes\operatorname{id})
(V):\omega\in{\mathcal B}({\mathcal H})_*\}}^w\,\bigl(\subseteq{\mathcal B}
({\mathcal H})\bigr)$.

Next, let us consider the {\em dual quantum group\/}.  Working with the other leg 
of the multiplicative unitary operator $W$, we define:
$$
\hat{M}=\overline{\bigl\{(\omega\otimes\operatorname{id})(W):\omega
\in{\mathcal B}({\mathcal H})_*\bigr\}}^w\,\bigl(\subseteq{\mathcal B}
({\mathcal H})\bigr).
$$
This is indeed shown to be a von Neumann algebra.  We can define a comultiplication 
on it, by $\hat{\Delta}(y)=\Sigma W(y\otimes1)W^*\Sigma$, for all $y\in\hat{M}$. 
Here, $\Sigma$ is the flip map on ${\mathcal H}\otimes{\mathcal H}$, and defining 
the dual comultiplication in this way makes it ``flipped'', unlike in the purely 
algebraic settings (See the remark following Proposition~\ref{pairing} for more 
discussion.).  This is done for technical reasons, so that it is simpler to work 
with the multiplicative unitary operator.  

The general theory assures that $(\hat{M},\hat{\Delta})$ is again a von Neumann 
algebraic quantum group, together with appropriate Haar weights $\hat{\varphi}$ 
and $\hat{\psi}$.  By taking the norm completion, we can consider the $C^*$-algebraic 
quantum group $(\hat{A},\hat{\Delta})$.  The operator $\hat{W}=\Sigma W^*\Sigma$ 
is easily seen to be the multiplicative unitary for $(\hat{M},\hat{\Delta})$. 
It turns out that $W\in M\otimes\hat{M}$ and $\hat{W}\in\hat{M}\otimes M$.

The left Haar weight $\hat{\varphi}$ on $(\hat{M},\hat{\Delta})$ is characterized 
by the GNS map $\hat{\Lambda}:{\mathfrak N}_{\hat{\varphi}}\to{\mathcal H}$, which 
is given by the following (See Proposition~8.14 of \cite{KuVa}):
\begin{equation}\label{(dualHaar)}
\bigl\langle\hat{\Lambda}\bigl((\omega\otimes\operatorname{id})(W),\Lambda(x)
\bigr\rangle=\omega(x^*).
\end{equation}
For this formula to make sense, we need $\omega\in{\mathcal B}({\mathcal H})_*$ 
to have $L\ge0$ such that $|\omega(x^*)|\le L\bigl\|\Lambda(x)\bigr\|$ for all 
$x\in{\mathfrak N}_{\varphi}$.  It is known that for such linear forms $\omega$, 
the elements $(\omega\otimes\operatorname{id})(W)$ form a core for $\hat{\Lambda}$.
See \cite{KuVa}, \cite{KuVavN}. 

The other structure maps for $(\hat{M},\hat{\Delta})$ are defined as before, including 
the modular operator $\hat{\nabla}$, the modular conjugation $\hat{J}$, and the 
antipode $\hat{S}$.  As for the antipode map $\hat{S}$, a similar characterization 
as in equation~\eqref{(antipode)} exists, with $\hat{W}=\Sigma W^*\Sigma$ now being 
the multiplicative unitary.  Namely, $\hat{S}\bigl((\omega\otimes\operatorname{id})(W^*)
\bigr)=(\omega\otimes\operatorname{id})(W)$.  The unitary antipode and the scaling 
group can be also found, giving us the polar decomposition $\hat{S}=\hat{R}
\hat{\tau}_{\frac{i}{2}}$. 

The modular conjugations $J$ and $\hat{J}$ are closely related with the antipode maps. 
In fact, it is known that $R(x)=\hat{J}x^*\hat{J}$, for $x\in M$ and $\hat{R}(y)
=Jy^*J$, for $y\in\hat{M}$.  It is also known that $\hat{J}J=\nu^{i/4}J\hat{J}$ 
(where $\nu$ is the ``scaling constant''), and that $W^*=(\hat{J}\otimes J)W(\hat{J}
\otimes J)$, and $V=(\hat{J}\otimes\hat{J})\Sigma W^*\Sigma(\hat{J}\otimes\hat{J})$. 
We have: $V\in\hat{M}'\otimes M$, where $\hat{M}'$ is the commutant of $\hat{M}$, 
with the opposite product.  See \cite{KuVavN} and \cite{VDoamp}, for further results 
on the relationships between various operators.

Repeating the whole process again, we can also construct the dual $(\hat{\hat{M}},
\hat{\hat{\Delta}})$ of $(\hat{M},\hat{\Delta})$.  An important result is the 
{\em generalized Pontryagin duality\/}, which says that $(\hat{\hat{M}},
\hat{\hat{\Delta}})\cong(M,\Delta)$.

We wrap up the subsection here.  For further details, we refer the reader to 
the fundamental papers on the subject: \cite{BS}, \cite{Wr7}, \cite{KuVa}, 
\cite{KuVavN}, \cite{MNW}, \cite{VDvN}.

\subsection{The dual pairing}
Suppose we have a mutually dual pair of quantum groups $(M,\Delta)$ and $(\hat{M},
\hat{\Delta})$.  Let $W$ be the associated multiplicative unitary operator.  The
dual pairing exists between $M$ and $\hat{M}$, but unlike in the (purely algebraic) 
cases of finite-dimensional Hopf algebras or multiplier Hopf algebras, the pairing 
map is defined only at the level of certain dense subalgebras of $M$ and $\hat{M}$. 
To be more specific, consider the subsets ${\mathcal A}$ ($\subseteq M$) and 
$\hat{\mathcal A}$ ($\subseteq\hat{M}$), defined by
$$
{\mathcal A}=\bigl\{(\operatorname{id}\otimes\omega)(W):
\omega\in M_*\bigr\}
$$
and
$$
\hat{\mathcal A}=\bigl\{(\omega'\otimes\operatorname{id})(W):
\omega'\in\hat{M}_*\bigr\}.
$$
By the general theory, it is known (see \cite{BS}, \cite{KuVavN}) that the spaces
${\mathcal A}$ and $\hat{\mathcal A}$ are actually (dense) subalgebras of $M$ and 
$\hat{M}$.  The dual pairing exists between $\hat{\mathcal A}$ and ${\mathcal A}$: 
That is, for $b=(\omega\otimes\operatorname{id})(W)\in\hat{\mathcal A}$ and 
$a=(\operatorname{id}\otimes\theta)(W)\in{\mathcal A}$, we have:
\begin{equation}\label{(pairing)}
\langle b\,|\,a\rangle=\bigl\langle(\omega\otimes\operatorname{id})(W)
\,|\,(\operatorname{id}\otimes\theta)(W)\bigr\rangle:=(\omega\otimes\theta)(W)
=\omega(a)=\theta\bigl(b).
\end{equation}
This definition is suggested by \cite{BS}.  The properties of this pairing map is 
given below:

\begin{prop}\label{pairing}
Let $(M,\Delta)$ and $(\hat{M},\hat{\Delta})$ be the dual pair of locally compact 
quantum groups, and let ${\mathcal A}$ and $\hat{\mathcal A}$ be their dense 
subalgebras, as defined above.  Then the map $\langle\ \,|\,\ \rangle:\hat{\mathcal A}
\times{\mathcal A}\to\mathbb{C}$, given by equation \eqref{(pairing)}, is a valid 
dual pairing.  Moreover, we have:
\begin{enumerate}
\item $\langle b_1b_2\,|\,a\rangle=\bigl\langle b_1\otimes b_2\,|\,\Delta(a)\bigr\rangle$,
for $a\in{\mathcal A}$, $b_1,b_2\in\hat{\mathcal A}$.
\item $\langle b\,|\,a_1a_2\rangle=\langle\hat{\Delta}^{\operatorname{cop}}(b)
\,|\,a_1\otimes a_2\rangle$, for $a_1,a_2\in{\mathcal A}$, $b\in\hat{\mathcal A}$.
\item $\bigl\langle b\,|\,S(a)\bigr\rangle=\bigl\langle\hat{S}^{-1}(b)\,|\,a\bigr\rangle$, 
for $a\in{\mathcal A}$, $b\in\hat{\mathcal A}$.
\end{enumerate}
\end{prop}
\begin{rem}
Bilinearity of $\langle\ \,|\,\ \rangle$ is obvious, and the proof of the three properties 
is straightforward.  See, for instance, Proposition~4.2 of \cite{BJKqfourier}.  Except for 
the appearance of the co-opposite comultiplication $\hat{\Delta}^{\operatorname{cop}}$ 
in (2), the proposition shows that $\langle\ \,|\,\ \rangle$ is a suitable dual pairing 
map that generalizes the pairing map on (finite-dimensional) Hopf algebras.  The difference 
is that in purely algebraic frameworks (Hopf algebras, QUE algebras, or even multiplier 
Hopf algebras), the dual comultiplication on $H'$ is simply defined by dualizing the product 
on $H$ via the natural pairing map between $H$ and $H'$.  Whereas in our case, the pairing 
is best defined using the multiplicative unitary operator.  It turns out that defining 
as we have done the dual comultiplication as ``flipped'' makes things to become technically 
simpler, even with (2) causing minor annoyance.
\end{rem}

Meanwhile, let us quote below an alternative description given in \cite{BJKqfourier} 
of this pairing map, using the Haar weights and the generalized Fourier transform. 
The new descriptions are only valid on certain subspaces $D\subseteq{\mathcal A}$ and 
$\hat{D}\subseteq\hat{\mathcal A}$, but $D$ and $\hat{D}$ are dense subalgebras in $M$ 
and $\hat{M}$ respectively, and form cores for the antipode maps $S$ and $\hat{S}$.

\begin{theorem}\label{pairingthm}
Let $D\subseteq{\mathcal A}$ and $\hat{D}\subseteq\hat{\mathcal A}$ be the 
dense subalgebras as defined in Section~4 of \cite{BJKqfourier}.  Then:
\begin{enumerate}
  \item For $a\in D$, its Fourier transform is defined by
$$
{\mathcal F}(a):=(\varphi\otimes\operatorname{id})\bigl(W(a\otimes1)\bigr).
$$
  \item For $b\in\hat{D}$, the inverse Fourier transform is defined by
$$
{\mathcal F}^{-1}(b):=(\operatorname{id}\otimes\hat{\varphi})
\bigl(W^*(1\otimes b)\bigr).
$$
  \item The dual pairing map $\langle\ \,|\,\ \rangle:\hat{\mathcal A}
\times{\mathcal A}\to\mathbb{C}$ given in Proposition~\ref{pairing} takes 
the following form, if we restrict it to the level of $D$ and $\hat{D}$:
\begin{align}
\langle b\,|\,a\rangle&
=\bigl\langle\hat{\Lambda}(b),\Lambda(a^*)\bigr\rangle
=\varphi\bigl(a{\mathcal F}^{-1}(b)\bigr)
=\hat{\varphi}\bigl({\mathcal F}(a^*)^*b\bigr)  \notag \\
&=(\varphi\otimes\hat{\varphi})\bigl[(a\otimes1)W^*(1\otimes b)\bigr].
\notag
\end{align}
\end{enumerate}
\end{theorem}

\begin{rem}
Here, $\varphi$ and $\hat{\varphi}$ are the left invariant Haar weights
for $(M,\Delta)$ and $(\hat{M},\hat{\Delta})$, while $\Lambda$ and $\hat
{\Lambda}$ are the associated GNS maps.  The maps ${\mathcal F}$ and 
${\mathcal F}^{-1}$ are actually defined in larger subspaces, but we 
restricted the domains here to $D$ and $\hat{D}$, for convenience.
As in the classical case, the Fourier inversion theorem holds:
$$
{\mathcal F}^{-1}\bigl({\mathcal F}(a)\bigr)=a, \ a\in D,\qquad
{\text {and}}\qquad
{\mathcal F}\bigl({\mathcal F}^{-1}(b)\bigr)=b, \ b\in\hat{D}.
$$
See \cite{BJKqfourier} for more careful discussion on all these, including 
the definition of the Fourier transform and the proof of the result on the 
dual pairing.
\end{rem}

\section{The quantum double}

The quantum double construction was originally introduced by Drinfeld \cite{Dr},
in the Hopf algebra framework.  The notion can be extended to the setting
of locally compact quantum groups.  See \cite{Ya} (also see \cite{BJKqdouble},
and some earlier results in \cite{PW} and Section~8 of \cite{BS}).  Some 
different formulations exist, but all of them are special cases of a more 
generalized notion of a {\em double crossed product construction\/} developed 
recently by Baaj and Vaes \cite{BjVa}.  While we do not plan to go to the 
full generality as in that paper, let us give here the definition adapted 
from \cite{BjVa}.

Let $(N,\Delta_N)$ be a locally compact quantum group, and let $W_N$ 
be its multiplicative unitary operator.  Write $(M_1,\Delta_1)=
(N,\Delta_N^{\operatorname{cop}})$ and $(M_2,\Delta_2)=(\hat{N},
\hat{\Delta}_N)$.  Suggested by Proposition~8.1 of \cite{BjVa}, consider 
the operators $K$ and $\hat{K}$ on ${\mathcal H}\otimes{\mathcal H}$:
$$
K=W_N(\hat{J}_1\otimes J_2)W_N^*,\qquad\hat{K}=W_N(J_1\otimes\hat{J}_2)W_N^*,
$$
where $J_1$, $\hat{J}_1$, $J_2$, $\hat{J}_2$ are the modular conjugations for
$M_1$, $\hat{M}_1$, $M_2$, $\hat{M}_2$.  In our case, we would actually have: 
$\hat{J}_1=J_2$ and $\hat{J}_2=J_1$.  Next, following Notation~3.2 of \cite{BjVa}, 
write:
$$
Z=K\hat{K}(\hat{J}_1J_1\otimes\hat{J}_2J_2).
$$
Then on ${\mathcal H}\otimes{\mathcal H}\otimes{\mathcal H}\otimes{\mathcal H}$,
define the unitary operator:
\begin{equation}\label{(Wm)}
W_m=(\Sigma V_1^*\Sigma)_{13}Z^*_{34}W_{2,24}Z_{34},
\end{equation}
where $V_1$ (right regular representation of $M_1$) and $W_2$ (left regular 
representation of $M_2$) are multiplicative unitary operators associated 
with $M_1$ and $M_2$.  By Proposition~3.5 and Theorem~5.3 of \cite{BjVa},
the operator $W_m$ is a multiplicative unitary operator, and it gives rise 
to a locally compact quantum group $(M_m,\Delta_m)$.  This is the ``double 
crossed product'' (in the sense of Baaj and Vaes \cite{BjVa}) of $(M_1,\Delta_1)$ 
and $(M_2,\Delta_2)$, and is to be called in Definition~\ref{qdouble} below 
as the dual of the quantum double.

\begin{defn}\label{qdouble}
Let $(N,\Delta_N)$ be a locally compact quantum group, with $W_N$ (``left regular 
representation'') and $V_N$ (``right regular representation'') being the associated 
multiplicative unitary operators.  In addition, denote by $J_N$, $\hat{J}_N$, 
$S_N$, $\varphi_N$, ... the relevant structure maps. 

Let $(M_1,\Delta_1)=(N,\Delta_N^{\operatorname{cop}})$, with the multiplicative 
unitary $W_1=\Sigma V_N^*\Sigma$.  We have: $J_1=J_N$ and $\hat{J}_1=\hat{J}_N$.
Also $V_1=(\hat{J}_1\otimes\hat{J}_1)\Sigma W_1^*\Sigma(\hat{J}_1\otimes
\hat{J}_1)$.  Since $J_1^2=\hat{J}_1^2=I_{\mathcal H}$, it becomes: $V_1=\Sigma 
W_N^*\Sigma$.  Meanwhile, let $(M_2,\Delta_2)=(\hat{N},\hat{\Delta}_N)$, which 
is associated with $W_2=\Sigma W_N^*\Sigma$.  We have: $J_2=\hat{J}_N$ and 
$\hat{J}_2=J_N$.  Using these ingredients, construct the multiplicative unitary 
operator $W_m\in{\mathcal B}({\mathcal H}\otimes{\mathcal H}\otimes{\mathcal H}
\otimes{\mathcal H})$, as given in equation~\eqref{(Wm)}.  Then:
\begin{enumerate}
  \item The {\em Drinfeld quantum double\/} is $(N_D,\Delta_D)$, given by 
the multiplicative unitary operator $W_D=\Sigma_{13}\Sigma_{24}W_m^*
\Sigma_{24}\Sigma_{13}$.  That is,
$$
N_D=\overline{\{(\operatorname{id}\otimes\operatorname{id}\otimes\Omega)(W_D):
\Omega\in{\mathcal B}({\mathcal H}\otimes{\mathcal H})_*\}}^w\,\bigl(\subseteq
{\mathcal B}({\mathcal H}\otimes{\mathcal H})\bigr),
$$
with the comultiplication $\Delta_D:N_D\to N_D\otimes N_D$, defined by
$\Delta_D(x):={W_D}^*(1\otimes1\otimes x)W_D$, for $x\in N_D$.
  \item The dual of the quantum double is $(\widehat{N_D},\widehat{\Delta_D})$,
determined by $W_m$.  Namely,
$$
\widehat{N_D}=\overline{\{(\operatorname{id}\otimes\operatorname{id}\otimes\Omega)
(W_m):\Omega\in{\mathcal B}({\mathcal H}\otimes{\mathcal H})_*\}}^w\,\bigl(\subseteq
{\mathcal B}({\mathcal H}\otimes{\mathcal H})\bigr),
$$
with the comultiplication $\widehat{\Delta_D}:\widehat{N_D}\to\widehat{N_D}
\otimes\widehat{N_D}$, given by $\widehat{\Delta_D}(y):={W_m}^*(1\otimes1
\otimes y)W_m$, for $y\in\widehat{N_D}$.
\end{enumerate}
By Theorem~5.3 of \cite{BjVa}, it is known that $N_D$ and $\widehat{N_D}$ 
are locally compact quantum groups, equipped with suitable Haar weights 
$\varphi_D$ and $\widehat{\varphi_D}$.
\end{defn}

Note here that we took the dual of $W_m$ in (1) to define the quantum double, 
so that our definition is more consistent with the ones given in the purely 
algebraic settings.  Because of this, our $(\widehat{N_D},\widehat{\Delta_D})$ 
is none other than $(M_m,\Delta_m)$, as defined in \cite{BjVa} (see that paper 
for details). 

While the Baaj and Vaes paper \cite{BjVa} discusses these in a more general 
setting, it is to be noted that the case of the quantum double of a locally 
compact quantum group is not explicitly studied there.  To be able to carry out 
the computations we have in mind, we need some specific details on the actual
structure of the quantum double and its dual.  This will be done in what follows.
Note that our setting here is still more general than the discussions given in 
\cite{PW}, \cite{BS}, \cite{Ya}, \cite{BJKqdouble}.

For convenience, we will just write from now on that $\Delta=\Delta_N$ and $W=W_N$. 
In our case, $V_1=\hat{W}=\Sigma W^*\Sigma$ and also $W_2=\hat{W}$, while $J=J_N
=J_1=\hat{J}_2$ and $\hat{J}=\hat{J}_N=\hat{J}_1=J_2$.  So we will have:
\begin{align}
\label{(Zdef)}
Z&=K\hat{K}(\hat{J}_1J_1\otimes\hat{J}_2J_2)=W(\hat{J}J\otimes\hat{J}J)W^*
(\hat{J}J\otimes J\hat{J})   \\
\label{(Wmdef)}
W_m&=(\Sigma V_1^*\Sigma)_{13}Z^*_{34}W_{2,24}Z_{34}=W_{13}Z^*_{34}\hat{W}_{24}Z_{34} \\
\label{(WDdef)}
W_D&=Z_{12}^*W_{24}Z_{12}\hat{W}_{13} 
\end{align}
We may occasionally be working at the $C^*$-algebra level.  In that case, we will 
consider $(A,\Delta)$ and $(\hat{A},\hat{\Delta})$, and the quantum double will be 
written as $(A_D,\Delta_D)$, and its dual $(\widehat{A_D},\widehat{\Delta_D})$. 
We just need to work with the same multiplicative unitary operators but replace 
the weak completions above to the norm completions. 

Let us begin first with $(\widehat{N_D},\widehat{\Delta_D})=(M_m,\Delta_m)$. 
See \cite{BjVa} for details. 

\begin{prop}\label{dualqdouble}
As a von Neumann algebra, we have: $\widehat{N_D}=N\otimes\hat{N}$, while 
the comultiplication $\widehat{\Delta_D}:\widehat{N_D}\to\widehat{N_D}\otimes
\widehat{N_D}$ is characterized as follows:
$$
\widehat{\Delta_D}=(\operatorname{id}\otimes\sigma\circ m\otimes\operatorname{id})
(\Delta_1^{\operatorname{cop}}\otimes\Delta_2)=(\operatorname{id}\otimes\sigma
\circ m\otimes\operatorname{id})(\Delta\otimes\hat{\Delta}).
$$
Here $\sigma:N\otimes\hat{N}\to\hat{N}\otimes N$ is the flip map, and 
$m:N\otimes\hat{N}\to N\otimes\hat{N}$ is the twisting map defined by 
$m(z)=ZzZ^*$.
\end{prop}

Its $C^*$-algebraic counterpart is rather tricky to describe.  In general, 
unless $W_D$ is regular (in the sense of Baaj and Skandalis \cite{BS}), 
it may be possible that $\widehat{A_D}\ne A\otimes\hat{A}$.  See discussion 
given in Section~9 of \cite{BjVa}.  Meanwhile, the description of the 
comultiplication $\widehat{\Delta_D}$ given above enables us to prove 
the following Lemma, which will be useful later:

\begin{lem}\label{lemmaW}
Let $W=W_N$, $\hat{W}=\Sigma W^*\Sigma$, $Z$ be the operators defined earlier. 
Then we have:
$$
Z_{34}Z_{12}^*W_{24}Z_{12}\hat{W}_{13}=\hat{W}_{13}Z_{12}^*W_{24}Z_{12}Z_{34}.
$$
\end{lem}

\begin{proof}
Since $W_m\in N_D\otimes\widehat{N_D}$ is the multiplicative unitary operator 
giving rise to the comultiplication $\widehat{\Delta_D}$, we should have (see 
\cite{BS}):
\begin{equation}\label{(W_mregularrep)}
(\widehat{\Delta_D}\otimes\operatorname{id})(W_m)=W_{m,13}W_{m,23}.
\end{equation}
From the definition of $W_m$ given in equation~\eqref{(Wmdef)}, the right side 
becomes:
$$
W_{m,13}W_{m,23}=W_{15}Z_{56}^*\hat{W}_{26}Z_{56}W_{35}
Z_{56}^*\hat{W}_{46}Z_{56}.
$$
Meanwhile, remembering that $\hat{\Delta}(b)=\hat{W}^*(1\otimes b)\hat{W}$ (for 
$b\in\hat{A}$) and that $\Delta(a)=W^*(1\otimes a)W$ (for $a\in A$), we have:
\begin{align}
(\Delta\otimes\hat{\Delta}\otimes\operatorname{id})(W_m)
&=(\Delta\otimes\hat{\Delta}\otimes\operatorname{id})\bigl[W_{13}Z_{34}^*\hat{W}_{24}
Z_{34}\bigr]   \notag \\
&=\bigl[W_{12}^*W_{25}W_{12}\bigr]Z_{56}^*\bigl[\hat{W}_{34}^*\hat{W}_{46}\hat{W}_{34}
\bigr]Z_{56}    \notag \\
&=W_{15}W_{25}Z_{56}^*\hat{W}_{36}\hat{W}_{46}Z_{56}  \notag \\
&=W_{15}W_{25}[Z_{56}^*\hat{W}_{36}Z_{56}][Z_{56}^*\hat{W}_{46}Z_{56}].
\notag
\end{align}
In the third equality, we used the pentagon relations for $W$ and for $\hat{W}$ 
(being multiplicative unitaries).  So we have:
\begin{align}
(\widehat{\Delta_D}\otimes\operatorname{id})(W_m)&=\bigl((\operatorname{id}\otimes
\sigma\circ m\otimes\operatorname{id})(\Delta\otimes\hat{\Delta})\bigr)(W_m)  \notag \\
&=Z_{32}W_{15}W_{35}[Z_{56}^*\hat{W}_{26}Z_{56}][Z_{56}^*\hat{W}_{46}Z_{56}]Z_{32}^*  
\notag \\
&=W_{15}Z_{32}W_{35}[Z_{56}^*\hat{W}_{26}Z_{56}]Z_{32}^*[Z_{56}^*\hat{W}_{46}Z_{56}].
\notag
\end{align}
Therefore, the equation~\eqref{(W_mregularrep)} now becomes (after obvious cancellations
and then multiplying $Z_{32}^*$ to both sides):
$$
W_{35}[Z_{56}^*\hat{W}_{26}Z_{56}]Z_{32}^*=Z_{32}^*Z_{56}^*\hat{W}_{26}Z_{56}W_{35}.
$$
Re-numbering the legs (legs 2,3,5,6 to become 4,3,1,2), we have:
$$
W_{31}Z_{12}^*\hat{W}_{42}Z_{12}Z_{34}^*=Z_{34}^*Z_{12}^*\hat{W}_{42}Z_{12}W_{31}.
$$
Now taking the adjoints from both sides, it becomes:
$$
Z_{34}Z_{12}^*\hat{W}_{42}^*Z_{12}W_{31}^*=W_{31}^*Z_{12}^*\hat{W}_{42}^*Z_{12}Z_{34}.
$$
Since $\hat{W}=\Sigma W^*\Sigma$, the result of Lemma follows immediately.
\end{proof}

Let us now turn our attention to $(N_D,\Delta_D)$.  We will give a more concrete 
realization of $N_D$ (in Proposition~\ref{A_Dalg}), as well as its coalgebra 
structure (in Proposition~\ref{A_Dcoalg}).  See also Theorem~5.3 of \cite{BjVa}.

\begin{prop}\label{A_Dalg}
Define $\pi:N\to{\mathcal B}({\mathcal H}\otimes{\mathcal H})$ and $\pi':
\hat{N}\to{\mathcal B}({\mathcal H}\otimes{\mathcal H})$ by
$$
\pi(f):=Z^*(1\otimes f)Z \qquad {\text {and}} \qquad \pi'(k):=k\otimes1.
$$
Then $N_D$ is the von Neumann algebra generated by the operators $\pi(f)\pi'(k)$, 
for $f\in N$, $k\in\hat{N}$.  The maps $\pi$ and $\pi'$ are in fact 
$W^*$-algebra homomorphisms.  Namely,
$$
\pi:N\to N_D\qquad{\text {and}}\qquad\pi':\hat{N}\to N_D.
$$
\end{prop}

\begin{proof}
Recall from equation~\eqref{(WDdef)} that $W_D=Z_{12}^*W_{24}Z_{12}\hat{W}_{13}$.
So for $\omega,\omega'\in{\mathcal B}({\mathcal H})_*$, we have:
\begin{align}
(\operatorname{id}\otimes\operatorname{id}\otimes\omega\otimes\omega')(W_D)
&=(\operatorname{id}\otimes\operatorname{id}\otimes\omega\otimes\omega')
(Z_{12}^*W_{24}Z_{12}\hat{W}_{13}) \notag \\
&=Z^*\bigl[1\otimes(\operatorname{id}\otimes\omega')(W)\bigr]Z
\bigl[(\operatorname{id}\otimes\omega)(\hat{W})\otimes1\bigr]=\pi(f)\pi'(k),
\notag
\end{align}
where $f=(\operatorname{id}\otimes\omega')(W)$ and $k=(\operatorname{id}
\otimes\omega)(\hat{W})$.  This makes sense, because $W\in N\otimes\hat{N}$ 
and $\hat{W}\in\hat{N}\otimes N$.  Recall the discussion in Section~2 above 
or Proposition~2.15 of \cite{KuVavN}.  In fact, the operators $(\operatorname{id}
\otimes\omega')(W)$, $\omega'\in{\mathcal B}({\mathcal H})_*$, generate the 
von Neumann algebra $N$; while the operators $(\operatorname{id}\otimes\omega)
(\hat{W})$, $\omega\in{\mathcal B}({\mathcal H})_*$, generate $\hat{N}$.

Since the operators $(\operatorname{id}\otimes\operatorname{id}\otimes\omega
\otimes\omega')(W_D)$ generate $N_D$ by Definition~\ref{qdouble}, the claim 
of the proposition is proved.  The second part of the proposition is obvious 
from the definitions.
\end{proof}

\begin{rem}
For future computation purposes, we will from now on regard $N_D$ to be the 
von Neumann algebra generated by the operators $\pi'(k)\pi(f)$, (for $f\in M$, 
$k\in\hat{M}$).  This is of course true, given the results of the previous 
proposition.  To be more specific, write:
\begin{equation}\label{Pi}
\Pi(k\otimes f):=\pi'(k)\pi(f),\qquad f\in N,\ k\in\hat{N}.
\end{equation}
Then we have: $N_D=\overline{\bigl\{\Pi(k\otimes f):f\in N,k\in\hat{N}\bigr\}}^w$.
Its $C^*$-algebraic counterpart is: $A_D=\overline{\bigl\{\Pi(k\otimes f):
f\in A,k\in\hat{A}\bigr\}}^{\|\ \|}$. 
\end{rem}

\begin{prop}\label{A_Dcoalg}
For $f\in N$ and $k\in\hat{N}$, we have:
\begin{align}
\Delta_D\bigl(\Pi(k\otimes f)\bigr)&=\Delta_D\bigl(\pi'(k)\pi(f)\bigr)
=\bigl[(\pi'\otimes\pi')\bigl(\hat{\Delta}k\bigr)\bigr]\bigl[(\pi\otimes\pi)
(\Delta f)\bigr]    \notag \\
&=(\Pi\otimes\Pi)\left(\sum k_{(1)}\otimes f_{(1)}\otimes k_{(2)}\otimes f_{(2)}\right).
\notag
\end{align}
\end{prop}

\begin{proof}
In the second line, we used the Sweedler's notation (see \cite{Mo}), 
where we write: $\Delta f=\sum f_{(1)}\otimes f_{(2)}$.
For computation, observe that
\begin{align}
\Delta_D\bigl(\pi'(k)\pi(f)\bigr)&={W_D}^*(1\otimes1\otimes\pi'(k)\pi(f))W_D
\notag \\
&=\bigl[{W_D}^*\bigl(1\otimes1\otimes\pi'(k)\bigr)W_D\bigr]\bigl[{W_D}^*
\bigl(1\otimes1\otimes\pi(f)\bigr)W_D\bigr].  \notag
\end{align}
Remembering the definitions of $W_D$ and $\pi'$ and $\pi$, we have:
\begin{align}
{W_D}^*\bigl(1\otimes1\otimes\pi'(k)\bigr)W_D
&=\hat{W}_{13}^*Z_{12}^*W_{24}^*Z_{12}(1\otimes1\otimes k\otimes1)
Z_{12}^*W_{24}Z_{12}\hat{W}_{13}  \notag \\
&=\hat{W}_{13}^*(1\otimes1\otimes k\otimes1)\hat{W}_{13}
=\bigl[\hat{\Delta}(k)\big]_{13}=(\pi'\otimes\pi')(\hat{\Delta}k).
\notag
\end{align}
Meanwhile, by Lemma~\ref{lemmaW}, we have:
\begin{align}
{W_D}^*\bigl(1\otimes1\otimes\pi(f)\bigr)W_D
&=\hat{W}_{13}^*Z_{12}^*W_{24}^*Z_{12}\bigl[Z^*(1\otimes f)Z\bigr]_{34}
Z_{12}^*W_{24}Z_{12}\hat{W}_{13}   \notag \\
&=Z_{34}^*Z_{12}^*W_{24}^*Z_{12}\hat{W}_{13}^*(1\otimes1\otimes1\otimes f)
\hat{W}_{13}Z_{12}^*W_{24}Z_{12}Z_{34}   \notag \\
&=Z_{34}^*Z_{12}^*W_{24}^*(1\otimes1\otimes1\otimes f)W_{24}Z_{12}Z_{34}
\notag \\
&=Z_{34}^*Z_{12}^*\bigl[\Delta(f)\bigr]_{24}Z_{12}Z_{34}=(\pi\otimes\pi)
(\Delta f).
\notag
\end{align}
Combining these two results, we prove the proposition.
\end{proof}

\begin{rem}
From the proof above, we see clearly that \ $(\pi\otimes\pi)\circ\Delta
=\Delta_D\circ\pi$, and that \ $(\pi'\otimes\pi')\circ\hat{\Delta}=\Delta_D
\circ\pi'$.  From these observations, we see that the ${}^*$-homomorphisms 
$\pi$ and $\pi'$ defined earlier are also coalgebra homomorphisms.
\end{rem}

As noted in Definition~\ref{qdouble}, the general theory assures us that 
$(N_D,\Delta_D)$ and $(\widehat{N_D},\widehat{\Delta_D})$ are indeed 
(mutually dual) locally compact quantum groups.  In particular, one can 
consider the (left) Haar weight $\varphi_D$ of $N_D$ and the (left) Haar 
weight $\widehat{\varphi_D}$ of $\widehat{N_D}$.  We give the descriptions 
of $\varphi_D$ and $\widehat{\varphi_D}$ below.

\begin{prop}\label{haarweight}
\begin{enumerate}
  \item The left Haar weight, $\varphi_D$, on $(N_D,\Delta_D)$ is characterized by 
the following:
$$
\varphi_D\bigl(\Pi(k\otimes f)\bigr)=\varphi_D\bigl(\pi'(k)\pi(f)\bigr)
=\hat{\varphi}(k)\varphi(f),\qquad {\text { for $f\in N$, $k\in\hat{N}$.}}
$$
  \item The left Haar weight, $\widehat{\varphi_D}$, on $\widehat{N_D}=N\otimes\hat{N}$ 
is as follows:
$$
\widehat{\varphi_D}(a\otimes b)=\varphi(a)\hat{\psi}(b).
$$
\end{enumerate}
\end{prop}

\begin{proof}
For (2), concerning the Haar weight on $(\widehat{N_D},\widehat{\Delta_D})$, see 
Theorem~5.3 of \cite{BjVa}, which says: $\varphi_m=\psi_1\otimes(\varphi_2)_{k_2}$. 
In our case, $\psi_1=\varphi$, because $(M_1,\Delta_1)=(N,\Delta^{\operatorname{cop}})$, 
while $\varphi_2=\hat{\varphi}$, because $(M_2,\Delta_2)=(\hat{N},\hat{\Delta})$. 
Moreover, our case being the ordinary quantum double of a locally compact quantum 
group, Proposition~8.1 of \cite{BjVa} indicates that $k_2=\delta_2$, the ``modular 
element'' of $(\hat{N},\hat{\Delta})$.  We thus have: $(\varphi_2)_{k_2}=\hat
{\varphi}_{\delta_2}=\hat{\psi}$.  

Consider now $\varphi_D$ given in (1).  To verify the left invariance, recall 
Proposition~\ref{A_Dcoalg} and compute:
\begin{align}
(\Omega\otimes\varphi_D)\bigl(\Delta_D(\Pi(k\otimes f))\bigr)&=\sum(\Omega\otimes\varphi_D)
\bigl((\Pi\otimes\Pi)(k_{(1)}\otimes f_{(1)}\otimes k_{(2)}\otimes f_{(2)})\bigr)
\notag \\
&=\sum\bigl[\Omega\bigl(\pi'(k_{(1)})\pi(f_{(1)})\bigr)
\varphi_D\bigl(\pi'(k_{(2)})\pi(a_{(2)})\bigr)\bigr]    \notag \\
&=\sum\bigl[\Omega\bigl((k_{(1)}\otimes1)Z^*(1\otimes f_{(1)})Z\bigr)
\hat{\varphi}(k_{(2)})\varphi(a_{(2)})\bigr].
\notag
\end{align}
Remembering the left invariance property of $\varphi$, which says: $\varphi\bigl((\omega
\otimes\operatorname{id})(\Delta f)\bigr)=\sum\bigl[\omega(f_{(1)})\varphi(f_{(2)})\bigr]
=\omega(1)\varphi(f)$, and similarly for $\hat{\varphi}$, we thus have:
$$
(\Omega\otimes\varphi_D)\bigl(\Delta_D(\Pi(k\otimes f))\bigr)
=\Omega(1\otimes1)\hat{\varphi}(k)\varphi(f)
=\Omega(1\otimes1)\varphi_D\bigl(\Pi(k\otimes f)\bigr),
$$
which is none other than the left invariance property for $\varphi_D$.  Though our 
proof is done only at the dense subalgebra level consisting of the $\Pi(k\otimes f)$,
it is sufficient, since we already know the existence of the unique Haar weight 
from the general theory.  By uniqueness, $\varphi_D$ described here must be the 
dual Haar weight on $(N_D,\Delta_D)$ corresponding to $\widehat{\varphi_D}$.
\end{proof}

Since we are not going to be prominently using them in this paper, we will skip the 
discussions on the right Haar weights and the antipode maps.  But let us just remind 
the reader that the antipode map $S_D$ can be obtained using the characterization 
given in equation~\eqref{(antipode)}, and similarly for $\widehat{S_D}$, working now 
with the operator $W_D$ instead.

\section{The twisting of the quantum double}

As is the case in the purely algebraic setting of QUE algebras \cite{Dr}, \cite{CP}, 
the quantum double $(A_D,\Delta_D)$ or $(N_D,\Delta_D)$ is equipped with a ``quantum 
universal $R$-matrix'' type operator ${\mathcal R}$.  Our plan is to use this 
operator to ``twist (deform)'' the comultiplication $\Delta_D$.

Let us begin by giving the definition and the construction of ${\mathcal R}$, 
in the operator algebra setting.  The approach is more or less the same as in 
Section~6 of \cite{BJKqdouble}, which was in turn adopted from Section~8 of 
\cite{BS}.  On the other hand, some modifications were necessary, because 
the current situation is more general than those in \cite{BS} and in 
\cite{BJKqdouble}, where the discussions were restricted to the case of 
so-called ``Kac systems''.  At present, the proof here seems to be the one 
that is being formulated in the most general setting.

\begin{lem}\label{lemmaR}
Let $W$, $\hat{W}$, $Z$ be the operators in ${\mathcal B}({\mathcal H}\otimes
{\mathcal H})$ defined earlier.  Then we have:
\begin{enumerate}
  \item $Z_{12}^*W_{45}W_{25}Z_{12}\hat{W}_{14}
=\hat{W}_{14}Z_{12}^*W_{25}W_{45}Z_{12}$
  \item $\hat{W}_{35}\hat{W}_{15}Z_{34}^*\hat{W}_{14}Z_{34}
=Z_{34}^*\hat{W}_{14}Z_{34}\hat{W}_{15}\hat{W}_{35}$
\end{enumerate}
\end{lem}

\begin{proof}
Recall from Lemma~\ref{lemmaW} that $Z_{34}Z_{12}^*W_{24}Z_{12}\hat{W}_{13}
=\hat{W}_{13}Z_{12}^*W_{24}Z_{12}Z_{34}$ or $Z_{12}^*Z_{34}W_{24}Z_{12}
\hat{W}_{13}=\hat{W}_{13}Z_{12}^*W_{24}Z_{34}Z_{12}$.  Recall now the definition 
of the operator $Z$ given in equation~\eqref{(Zdef)}, and write: $Z=WT$, where 
$T=(\hat{J}J\otimes\hat{J}J)W^*(\hat{J}J\otimes J\hat{J})$.  Then our equation 
above becomes (by writing $Z_{34}=W_{34}T_{34}$):
$$
Z_{12}^*W_{34}T_{34}W_{24}Z_{12}\hat{W}_{13}
=\hat{W}_{13}Z_{12}^*W_{24}W_{34}T_{34}Z_{12}.
$$
We can cancel out $T_{34}$ from both sides, because $T_{34}$ actually commutes 
with all the operators in the equation.  To see this, note that $\hat{J}J
=kJ\hat{J}$, for a constant $k$ (actually $k=\nu^{1/4}$, where $\nu>0$ is 
the ``scaling constant'' \cite{KuVa}, \cite{KuVavN}).  So we can write:
$$
T=k(J\otimes\hat{J})(\hat{J}\otimes J)W^*(\hat{J}\otimes J)(J\otimes\hat{J}).
$$
Since $W^*\in M\otimes\hat{M}$, we can see that $T\in M'\otimes\hat{M}'$.

So far we have: $Z_{12}^*W_{34}W_{24}Z_{12}\hat{W}_{13}=\hat{W}_{13}Z_{12}^*
W_{24}W_{34}Z_{12}$.  By re-numbering the legs (letting 3,4 to become 4,5), 
we obtain (1):
$$
Z_{12}^*W_{45}W_{25}Z_{12}\hat{W}_{14}=\hat{W}_{14}Z_{12}^*W_{25}W_{45}Z_{12}.
$$

Next, we re-write (1), using $\hat{W}=\Sigma W^*\Sigma$.  Then we have:
$$
Z_{12}^*\hat{W}_{54}^*\hat{W}_{52}^*Z_{12}\hat{W}_{14}
=\hat{W}_{14}Z_{12}^*\hat{W}_{52}^*\hat{W}_{54}^*Z_{12}.
$$
Apply $Z_{12}^*\hat{W}_{52}\hat{W}_{54}Z_{12}[\,\cdots\,]Z_{12}^*\hat{W}_{54}
\hat{W}_{52}Z_{12}$ to both sides.  Then:
$$
\hat{W}_{14}Z_{12}^*\hat{W}_{54}\hat{W}_{52}Z_{12}
=Z_{12}^*\hat{W}_{52}\hat{W}_{54}Z_{12}\hat{W}_{14},
$$
which is same as: $\hat{W}_{14}\hat{W}_{54}Z_{12}^*\hat{W}_{52}Z_{12}
=Z_{12}^*\hat{W}_{52}Z_{12}\hat{W}_{54}\hat{W}_{14}$.  Here, we re-number the legs
(letting 1,2,4,5 to become 3,4,5,1), and obtain (2):
$$
\hat{W}_{35}\hat{W}_{15}Z_{34}^*\hat{W}_{14}Z_{34}
=Z_{34}^*\hat{W}_{14}Z_{34}\hat{W}_{15}\hat{W}_{35}.
$$
\end{proof}

Lemma~\ref{lemmaR} above will be helpful in our proof of the next proposition,
which gives the description of our ``quantum $R$-matrix'' type operator
${\mathcal R}$.

\begin{prop}\label{rmatrix}
Let ${\mathcal R}\in{\mathcal B}\bigl(({\mathcal H}\otimes{\mathcal H})\otimes
({\mathcal H}\otimes{\mathcal H})\bigr)$ be the operator defined by ${\mathcal R}
=Z_{34}^*\hat{W}_{14}Z_{34}$.  The following properties hold:
\begin{enumerate}
  \item ${\mathcal R}\in M(A_D\otimes A_D)\,\subseteq N_D\otimes N_D$ and 
${\mathcal R}$ is unitary: ${\mathcal R}^{-1}={\mathcal R}^*$.
  \item We have: $(\Delta_D\otimes\operatorname{id})({\mathcal R})
={\mathcal R}_{13}{\mathcal R}_{23}$ and $(\operatorname{id}\otimes\Delta_D)
({\mathcal R})={\mathcal R}_{13}{\mathcal R}_{12}$.
  \item For any $x\in A_D$, we have: ${\mathcal R}\bigl(\Delta_D(x)\bigr)
{\mathcal R}^*=\Delta_D^{\operatorname{cop}}(x)$.
  \item The operator ${\mathcal R}$ satisfies the ``quantum Yang-Baxter equation 
(QYBE)'': Namely, ${\mathcal R}_{12}{\mathcal R}_{13}{\mathcal R}_{23}
={\mathcal R}_{23}{\mathcal R}_{13}{\mathcal R}_{12}$.
\end{enumerate}
\end{prop}

\begin{proof}
Here $M(B)$ denotes the multiplier algebra of a $C^*$-algebra $B$.

(1) Recall that $\hat{W}\in\hat{N}\otimes N$.  Therefore, by naturally extending 
the ${}^*$-homomorphisms $\pi$ and $\pi'$ defined in Proposition~\ref{A_Dalg}, 
we can see that ${\mathcal R}=(\pi'\otimes\pi)(\hat{W})\in N_D\otimes N_D$. 
Actually, noting that $\hat{W}\in M(\hat{A}\otimes A)$, we also see that 
${\mathcal R}\in M(A_D\otimes A_D)$.  Meanwhile, from the definitions of the 
operators involved, it is clear that ${\mathcal R}$ is unitary.

(2) Since ${\mathcal R}=(\pi'\otimes\pi)(\hat{W})$, we have:
\begin{align}
(\Delta_D\otimes\operatorname{id})({\mathcal R})&=(\Delta_D\otimes\operatorname{id})
\bigl((\pi'\otimes\pi)(\hat{W})\bigr)=(\pi'\otimes\pi'\otimes\pi)\bigl((\hat{\Delta}
\otimes\operatorname{id})(\hat{W})\bigr)  \notag \\
&=(\pi'\otimes\pi'\otimes\pi)(\hat{W}_{12}^*\hat{W}_{23}\hat{W}_{12})
=(\pi'\otimes\pi'\otimes\pi)(\hat{W}_{13}\hat{W}_{23})   \notag \\
&=\bigl[(\pi'\otimes\pi'\otimes\pi)(\hat{W}_{13})\bigr]
\bigl[(\pi'\otimes\pi'\otimes\pi)(\hat{W}_{23})\bigr]
={\mathcal R}_{13}{\mathcal R}_{23}.
\notag
\end{align}
The second equality is due to $\Delta_D\circ\pi'=(\pi'\otimes\pi')\circ\hat{\Delta}$
(see Proposition~\ref{A_Dcoalg}).  Since $\hat{W}\in M(\hat{A}\otimes A)$, the 
third equality follows from the definition of $\hat{\Delta}$.  The fourth equality 
is the pentagon equation for $\hat{W}$ (being multiplicative).  The fifth equality 
is just using the fact that $\pi'$ and $\pi$ are $C^*$-homomorphisms.

The proof for $(\operatorname{id}\otimes\Delta_D)({\mathcal R})={\mathcal R}_{13}
{\mathcal R}_{12}$ can be done in a similar way.  Just use the fact that for 
$a\in A$, we have: $\Delta a=W^*(1\otimes a)W=\hat{W}_{21}(1\otimes a)\hat{W}_{21}^*$,
and that $\Delta_D\circ\pi=(\pi\otimes\pi)\circ\Delta$.

(3) Recall from Section 2 that $(\operatorname{id}\otimes\omega)(\hat{W})\in\hat{A}$,
and $(\operatorname{id}\otimes\omega')(W)\in A$, for $\omega,\omega'\in{\mathcal B}
({\mathcal H})_*$, and that these operators generate $\hat{A}$ and $A$, respectively.
So consider $b=(\operatorname{id}\otimes\omega)(\hat{W})\in\hat{A}$ and compute.  Then:
\begin{align}
{\mathcal R}\bigl[\Delta_D\bigl(\pi'(b)\bigr)\bigr]&={\mathcal R}\bigl[(\pi'\otimes\pi')
(\hat{\Delta}b)\bigr]=(Z_{34}^*\hat{W}_{14}Z_{34})\bigl[\hat{W}_{13}^*(1\otimes1\otimes
b\otimes1)\hat{W}_{13}\bigr]   \notag \\
&=(\operatorname{id}\otimes\operatorname{id}\otimes\operatorname{id}\otimes\operatorname{id}
\otimes\omega)(Z_{34}^*\hat{W}_{14}Z_{34}\hat{W}_{13}^*\hat{W}_{35}\hat{W}_{13})
\notag \\
&=(\operatorname{id}\otimes\operatorname{id}\otimes\operatorname{id}\otimes\operatorname{id}
\otimes\omega)(Z_{34}^*\hat{W}_{14}Z_{34}\hat{W}_{15}\hat{W}_{35})   \notag \\
&=(\operatorname{id}\otimes\operatorname{id}\otimes\operatorname{id}\otimes\operatorname{id}
\otimes\omega)(\hat{W}_{35}\hat{W}_{15}Z_{34}^*\hat{W}_{14}Z_{34})  \notag \\
&=(\operatorname{id}\otimes\operatorname{id}\otimes\operatorname{id}\otimes\operatorname{id}
\otimes\omega)(\hat{W}_{31}^*\hat{W}_{15}\hat{W}_{31}Z_{34}^*\hat{W}_{14}Z_{34})  \notag \\
&=\bigl[\hat{\Delta}^{\operatorname{cop}}(b)\bigr]_{13}(Z_{34}^*\hat{W}_{14}Z_{34}) \notag \\
&=\bigl[(\pi'\otimes\pi')\bigl(\hat{\Delta}^{\operatorname{cop}}(b)\bigr)\bigr]{\mathcal R}
=\bigl[\Delta_D^{\operatorname{cop}}\bigl(\pi'(b)\bigr)\bigr]{\mathcal R}.
\notag 
\end{align}
The fourth and sixth equalities follow from the multiplicativity of $\hat{W}$, while the 
fifth equality is using Lemma~\ref{lemmaR} (2).  In the seventh equality, we used the 
fact that $\hat{\Delta}^{\operatorname{cop}}(b)=W(b\otimes 1)W^*=\hat{W}_{21}^*(b\otimes1)
\hat{W}_{21}$.

Next, consider $a=(\operatorname{id}\otimes\omega')(W)\in A$ and compute.  Then:
\begin{align}
{\mathcal R}\bigl[\Delta_D\bigl(\pi(a)\bigr)\bigr]
&=(Z_{34}^*\hat{W}_{14}Z_{34})\bigl[Z_{12}^*Z_{34}^*W_{24}^*(1\otimes1\otimes1\otimes a)W_{24}
Z_{12}Z_{34}\bigr]  \notag \\
&=(\operatorname{id}\otimes\operatorname{id}\otimes\operatorname{id}\otimes\operatorname{id}
\otimes\omega)(Z_{34}^*\hat{W}_{14}Z_{12}^*W_{24}^*W_{45}W_{24}Z_{12}Z_{34})
\notag \\
&=(\operatorname{id}\otimes\operatorname{id}\otimes\operatorname{id}\otimes\operatorname{id}
\otimes\omega)(Z_{34}^*\hat{W}_{14}Z_{12}^*W_{25}W_{45}Z_{12}Z_{34})   \notag \\
&=(\operatorname{id}\otimes\operatorname{id}\otimes\operatorname{id}\otimes\operatorname{id}
\otimes\omega)(Z_{34}^*Z_{12}^*W_{45}W_{25}Z_{12}\hat{W}_{14}Z_{34})  \notag \\
&=(\operatorname{id}\otimes\operatorname{id}\otimes\operatorname{id}\otimes\operatorname{id}
\otimes\omega)(Z_{34}^*Z_{12}^*W_{42}^*W_{25}W_{42}Z_{12}Z_{34}Z_{34}^*\hat{W}_{14}Z_{34})  
\notag \\
&=Z_{34}^*Z_{12}^*\bigl[\hat{W}_{24}(1\otimes a\otimes1\otimes1)\hat{W}_{24}^*\bigr]
Z_{12}Z_{34}(Z_{34}^*\hat{W}_{14}Z_{34})  \notag \\
&=\bigl[(\pi\otimes\pi)\bigl(\Delta^{\operatorname{cop}}(a)\bigr)\bigr]{\mathcal R}
=\bigl[\Delta_D^{\operatorname{cop}}\bigl(\pi(a)\bigr)\bigr]{\mathcal R}.
\notag 
\end{align}
This is essentially the same computation as the previous one.  The pentagon equation 
for $W$ is used in the fourth and sixth equalities.  The fifth equality is using 
Lemma~\ref{lemmaR} (1).  In the seventh and eighth equalities, note $\hat{W}=\Sigma
W^*\Sigma$ and also note $\Delta^{\operatorname{cop}}(a)=\Sigma W^*(1\otimes a)W\Sigma
=\hat{W}(a\otimes1)\hat{W}^*$.

Since it has been observed that $A_D$ is generated by the operators $\pi'(b)\pi(a)$, 
we conclude from the two results above (as well as the unitarity of ${\mathcal R}$) 
that: ${\mathcal R}\bigl[\Delta_D(x)\bigr]{\mathcal R}^*=\Delta_D^{\operatorname{cop}}(x)$, 
for any $x\in A_D$.

(4) The QYBE follows right away from (3) and (4).  In fact,
$$
{\mathcal R}_{12}{\mathcal R}_{13}{\mathcal R}_{23}={\mathcal R}_{12}\bigl[(\Delta_D
\otimes\operatorname{id})({\mathcal R})\bigr]=\bigl[(\Delta_D^{\operatorname{cop}}
\otimes\operatorname{id})({\mathcal R})\bigr]{\mathcal R}_{12}={\mathcal R}_{23}
{\mathcal R}_{13}{\mathcal R}_{12}.
$$
The first equality follows from (2); the second equality is from (3); and the third 
equality is from (2), with the legs 1 and 2 interchanged.
\end{proof}

As a quick consequence of Proposition~\ref{rmatrix}, we point out that ${\mathcal R}$ 
determines a certain ``left 2-cocycle'' (dual to the notion of same name in the Hopf 
algebra setting, introduced in Section~3 of \cite{Ludouble}).  While we do not need to 
give the definition of a 2-cocycle here, this means that we can {\em deform\/} (or twist) 
the comultiplication $\Delta_D$ by multiplying ${\mathcal R}$ from the left, and obtain 
a new map satisfying the coassociativity.  The result is given below:

\begin{prop}\label{RDeltacoassociativity}
Let ${}_{\mathcal R}\Delta:A_D\to M(A_D\otimes A_D)$ be defined by
$$
{}_{\mathcal R}\Delta(x):={\mathcal R}\Delta_D(x),\qquad {\text {for $x\in A_D$.}}
$$
Then ${}_{\mathcal R}\Delta$ satisfies the coassociativity: $({}_{\mathcal R}\Delta
\otimes\operatorname{id}){}_{\mathcal R}\Delta
=(\operatorname{id}\otimes{}_{\mathcal R}\Delta){}_{\mathcal R}\Delta$.
\end{prop}

\begin{proof}
The definition for ${}_{\mathcal R}\Delta$ makes sense, since ${\mathcal R}
\in M(A_D\otimes A_D)$.  Now for any $x\in A_D$, we have:
\begin{align}
({}_{\mathcal R}\Delta\otimes\operatorname{id}){}_{\mathcal R}\Delta(x)
&={\mathcal R}_{12}(\Delta_D\otimes\operatorname{id})\bigl({\mathcal R}
\Delta_D(x)\bigr)   \notag \\
&={\mathcal R}_{12}\bigl[(\Delta_D\otimes\operatorname{id})({\mathcal R})\bigr]
\bigl[(\Delta_D\otimes\operatorname{id})\bigl(\Delta_D(x)\bigr)\bigr]  \notag \\
&={\mathcal R}_{12}\bigl[{\mathcal R}_{13}{\mathcal R}_{23}\bigr]
\bigl[(\Delta_D\otimes\operatorname{id})\bigl(\Delta_D(x)\bigr)\bigr]  \notag \\
&={\mathcal R}_{23}\bigl[{\mathcal R}_{13}{\mathcal R}_{12}\bigr]
\bigl[(\operatorname{id}\otimes\Delta_D)\bigl(\Delta_D(x)\bigr)\bigr]   \notag \\
&={\mathcal R}_{23}\bigl[(\operatorname{id}\otimes\Delta_D)({\mathcal R})\bigr]
\bigl[(\operatorname{id}\otimes\Delta_D)\bigl(\Delta_D(x)\bigr)\bigr]   \notag \\
&={\mathcal R}_{23}(\operatorname{id}\otimes\Delta_D)\bigl({\mathcal R}
\Delta_D(x)\bigr)
=(\operatorname{id}\otimes{}_{\mathcal R}\Delta){}_{\mathcal R}\Delta(x).
\notag
\end{align}
In the second and sixth equalities, we used the fact that $\Delta_D$ is a 
$C^*$-homomorphism.  The third and fifth equalities used Proposition~\ref{rmatrix} (2). 
In the fourth equality, we used the QYBE and the coassociativity of $\Delta_D$.
\end{proof}

The coassociative map ${}_{\mathcal R}\Delta$ above is certainly a ``deformed 
$\Delta_D$''.  However, it should be noted that $(A_D,{}_{\mathcal R}\Delta)$ 
is not going to give us any valid quantum group.  For instance, it is impossible 
to define a suitable Haar weight.  And, ${}_{\mathcal R}\Delta$ is not even 
a ${}^*$-homomorphism.  On the other hand, considering that $\Delta_D$ is 
``dual'' to the algebra structure on $\widehat{A_D}$ (via $W_D$ and 
Proposition~\ref{pairing}), and since ${}_{\mathcal R}\Delta$ still carries 
a sort of a ``non-degeneracy'' (since $\Delta_D$ is a non-degenerate 
$C^*$-morphism and ${\mathcal R}$ is a unitary map), we may try to deform 
the algebra structure on $\widehat{A_D}$ by dualizing ${}_{\mathcal R}\Delta$. 
Formally, we wish to define on the vector space $\widehat{A_D}$ a new product 
$\times_{\mathcal R}$, given by
\begin{equation}\label{(deformedprod)}
\langle f\times_{\mathcal R}g\,|\,x\rangle
=\bigl\langle f\otimes g\,|\,{}_{\mathcal R}\Delta(x)\bigr\rangle,
\end{equation}
where $f,g\in\widehat{A_D}$ and $x\in A_D$.

The obvious trouble with this program is that $(A_D,{}_{\mathcal R}\Delta)$ is 
no longer a quantum group, which means that we do not have any multiplicative 
unitary operator that was essential in formulating the dual pairing in the case 
of locally compact quantum groups.  In the next two sections, we will try to 
make sense of the formal equation~\eqref{(deformedprod)}, and use it to construct 
a $C^*$-algebra (though not a quantum group) that can be considered as a 
``deformed $\widehat{A_D}$''.  Let us begin with the case of $A=C_{red}^*(G)$.

\section{The case of an ordinary group.  The Weyl algebra.}

For this section, let $G$ be an ordinary locally compact group, with a fixed 
left Haar measure $dx$.  Let $\nabla(x)$ denote the modular function.  Using 
the Haar measure, we can form the Hilbert space ${\mathcal H}=L^2(G)$.  We then 
construct two natural subalgebras, $N$ and $\hat{N}$ of ${\mathcal B}({\mathcal H})$, 
as follows.

First consider the von Neumann algebra $N={\mathcal L}(G)$, given by the left regular 
representation.  That is, for $a\in C_c(G)$, let $L_a\in{\mathcal B}({\mathcal H})$ 
be such that $L_a\xi(t)=\int a(z)\xi(z^{-1}t)\,dz$.  We take ${\mathcal L}(G)$ to be 
the $W^*$-closure of $L\bigl(C_c(G)\bigr)$.  Next consider $\hat{N}=L^{\infty}(G)$, 
where $b\in L^{\infty}(G)$ is viewed as the multiplication operator $\mu_b$ on 
${\mathcal H}=L^2(G)$, by $\mu_b\xi(t)=b(t)\xi(t)$.  These are well-known von 
Neumann algebras, and it is also known that we can give (mutually dual) quantum 
group structures on them.  We briefly review the results below.

Let $W\in{\mathcal B}({\mathcal H}\otimes{\mathcal H})={\mathcal B}\bigl(L^2(G\times G)
\bigr)$ be defined by $W\xi(s,t)=\xi(ts,t)$.  It is actually the dual (that is, 
$W=\Sigma W_G^*\Sigma$)  of the well-known multiplicative unitary operator $W_G$, 
defined by $W_G\xi(s,t)=\xi(s,s^{-1}t)$, and is therefore multiplicative \cite{BS}.
We can show without difficulty that
$$
N={\mathcal L}(G)=\overline{\bigl\{(\operatorname{id}\otimes\omega)(W):\omega
\in{\mathcal B}({\mathcal H})_*\bigr\}}^w,
$$
and the comultiplication on $N$ is given by $\Delta(x)=W^*(1\otimes x)W$, for 
$x\in N$.  For $a\in C_c(G)$, this reads: $(L\otimes L)_{\Delta a}\xi(s,t)
=\int a(z)\xi(z^{-1}s,z^{-1}t)\,dz$.  The antipode map $S:a\to S(a)$ is such that 
$\bigl(S(a)\bigr)(t)=\nabla(t^{-1})a(t^{-1})$, where $\nabla$ is the modular function.
The left Haar weight is given by $\varphi(a)=a(1)$, where $1=1_G$ is the group 
identity element.  In this way, we obtain a von Neumann algebraic quantum group 
$(N,\Delta)$, which is co-commutative .

Meanwhile, we can also show that:
$$
\hat{N}=L^{\infty}(G)=\overline{\bigl\{(\omega\otimes\operatorname{id})(W):\omega
\in{\mathcal B}({\mathcal H})_*\bigr\}}^w,
$$
and the comultiplication on $\hat{N}$ is given by $\hat{\Delta}(y)=\Sigma 
W(y\otimes1)W^*\Sigma$, for $y\in\hat{N}$.  In effect, this will give us 
$\hat{\Delta}b(s,t)=b(st)$, for $b\in L^{\infty}(G)$.  The antipode map 
$\hat{S}:b\to\hat{S}(b)$ is such that $\bigl(\hat{S}(b)\bigr)(t)=b(t^{-1})$, 
while the left Haar weight is just $\hat{\varphi}(b)=\int b(t)\,dt$.  In this 
way, $(\hat{N},\hat{\Delta})$ becomes a commutative von Neumann algebraic 
quantum group.

By considering the norm completions instead, we will have the $C^*$-algebraic 
quantum groups $A=C_{red}^*(G)$ and $\hat{A}=C_0(G)$.  Meanwhile, as in 
Proposition~\ref{pairing}, a dual pairing map can be considered at the 
level of certain dense subalgebras.  For convenience, let us consider 
$L\bigl(C_c(G)\bigr)\subseteq N$ and $\mu\bigl(C_c(G)\bigr)\subseteq\hat{N}$. 
The dual pairing defined by the multiplicative unitary operator $W$, as given 
in equation~\eqref{(pairing)} (or see Theorem~\ref{pairingthm}), becomes:
\begin{equation}\label{(Gpairing)}
\bigl\langle\mu_b\,|\,L_a\bigr\rangle=\int a(t)b(t^{-1})\,dt,
\end{equation}
for $\mu_b\in\mu\bigl(C_c(G)\bigr)$ and $L_a\in L\bigl(C_c(G)\bigr)$.
The proof is straightforward.

We now turn to find a more concrete description of the quantum double, $D(G)=A_D$, 
and its dual $\widehat{D(G)}$.  First, consider the operators $J$ and $\hat{J}$ on 
${\mathcal H}=L^2(G)$, which come from our knowledge of the involution and the 
antipode maps.
$$
J\xi(s)=\overline{\nabla(s^{-1})\xi(s^{-1})}, \qquad \hat{J}\xi(s)=\overline{\xi(s)}.
$$
Following the definitions given in Section~3, given by equations~\eqref{(Zdef)}, 
\eqref{(Wmdef)}, \eqref{(WDdef)}, construct the operator $Z\in{\mathcal B}
\bigl(L^2(G\times G)\bigr)$, as well as $W_m$ and $W_D$, which act on 
$L^2(G\times G\times G\times G)$.  We have:
\begin{align}
Z\xi(s,t)&=W(\hat{J}J\otimes\hat{J}J)W^*(\hat{J}J\otimes J\hat{J})\xi(s,t)
=\nabla(t^{-1})\xi(tst^{-1},t).   \notag \\
W_m\xi(s,t,s',t')&=W_{13}Z_{34}^*\hat{W}_{24}Z_{34}\xi(s,t,s',t')
=\nabla(t)\xi(s's,t,t^{-1}s't,t^{-1}t').    \notag \\
W_D\xi(s,t,s',t')&=\Sigma_{13}\Sigma_{24}W_m^*\Sigma_{24}\Sigma_{13}
\xi(s,t,s',t')    \notag \\
&=\nabla(t'^{-1})\xi(t'st'^{-1},t't,t's^{-1}t'^{-1}s',t').
\notag
\end{align}
Next, by using the results of Propositions~\ref{dualqdouble}, \ref{A_Dalg}, 
\ref{A_Dcoalg}, we can give below the descriptions for the quantum double 
and its dual:

\begin{prop}\label{D(G)}
Let $A=C_{red}^*(G)$ and $\hat{A}=C_0(G)$ be the (mutually dual) quantum groups 
associated with $G$, equipped with their natural structure maps described above.
Then:
\begin{enumerate}
  \item As a $C^*$-algebra, we have: 
$$
D(G)=\overline{\bigl\{\Pi(\mu_k\otimes L_f):f,k\in C_c(G)\bigr\}}^{\|\ \|}
\cong C_0(G)\rtimes_{\alpha}G,
$$
where $\alpha$ is the  conjugation action.
  \item The comultiplication on $D(G)$ is given by
$$
\Delta_D\bigl(\Pi(\mu_k\otimes L_f)\bigr)=\bigl[(\pi'\otimes\pi')
(\hat{\Delta}(\mu_k))\bigr]\bigl[(\pi\otimes\pi)(\Delta(L_f))\bigr].
$$
  \item As a $C^*$-algebra, we have: $\widehat{D(G)}=A\otimes\hat{A}
=C_{red}^*(G)\otimes C_0(G)$.
  \item The comultiplication on $\widehat{D(G)}$ is given by
$$
\widehat{\Delta_D}=(\operatorname{id}\otimes\sigma\circ m\otimes\operatorname{id})
(\Delta\otimes\hat{\Delta}),
$$
where $m(z)=ZzZ^*$, for $z\in M(A\otimes\hat{A})$.
\end{enumerate}
\end{prop}

\begin{proof}
Recall equation~\eqref{Pi} for the definition of $\Pi$, given in terms of the 
${}^*$-homomorphisms $\pi'$ and $\pi$ from Proposition~\ref{A_Dalg}.  
For (1), note that:
\begin{align}
\Pi(\mu_k\otimes L_f)\xi(s,t)&=\pi'(\mu_k)\pi(L_f)\xi(s,t)
=(\mu_k\otimes1)Z^*(1\otimes L_f)Z\xi(s,t)  \notag \\
&=\int\nabla(z)k(s)f(z)\xi(z^{-1}sz,z^{-1}t)\,dz.\label{PiD(G)}
\end{align}
If we write $\alpha_z\xi(s)=\xi(z^{-1}sz)$, $z\in G$, as the conjugation action, 
we can see without much difficulty from above that the $C^*$-algebra $D(G)$, 
which is generated by the operators $\Pi(\mu_k\otimes L_f)$, is isomorphic 
to the crossed product algebra $C_0(G)\rtimes_{\alpha}G$. [See any standard 
textbook on $C^*$-algebras, which contains discussion on crossed products.] 
By Proposition~\ref{A_Dcoalg}, we also know that the comultiplication on $D(G)$ 
is given as in (2).

In our case, being ``regular'', we do have: $\widehat{D(G)}=A\otimes\hat{A}$.
At the level of the functions in $C_c(G\times G)$, the multiplication on 
$\widehat{D(G)}=C_{red}^*(G)\otimes C_0(G)$ noted in (3) is reflected as follows:
\begin{equation}\label{(undeformedprod)}
\bigl[(a\otimes b)\times(a'\otimes b')\bigr](s,t)=\int a(z)b(t)a'(z^{-1}s)b'(t)\,dsdt.
\end{equation}
The description given in (4) of the comultiplcation $\widehat{\Delta_D}$ follows 
from Proposition~\ref{dualqdouble}.
\end{proof}

The next proposition describes the dual pairing map.  We may use equation~\eqref
{(pairing)}, but we instead give our proof using Theorem~\ref{pairingthm}.

\begin{prop}\label{pairingD(G)}
The dual pairing map is defined between the (dense) subalgebras $(L\otimes\mu)
\bigl(C_c(G\times G)\bigr)\subseteq\widehat{D(G)}$ and $\Pi\bigl((\mu\otimes L)
\bigl(C_c(G\times G)\bigr)\bigr)\subseteq D(G)$.  Applying Theorem~\ref{pairingthm}, 
we have:
\begin{align}
\bigl\langle L_a\otimes\mu_b\,|\,\Pi(\mu_k\otimes L_f)\bigr\rangle
&=(\varphi_D\otimes\widehat{\varphi_D})\bigl[(\Pi(\mu_k\otimes L_f)
\otimes1\otimes1)W_D^*(1\otimes1\otimes L_a\otimes\mu_b)\bigr]   \notag \\
&=\int\nabla(t)a(t^{-1}st)b(t^{-1})k(s)f(t)\,dsdt,
\notag
\end{align}
where $L_a,L_f\in L\bigl(C_c(G)\bigr)\subseteq A$ and $\mu_b,\mu_k
\in\mu\bigl(C_c(G)\bigr)\subseteq\hat{A}$.
\end{prop}

\begin{proof}
Recall from Proposition~\ref{haarweight} that the Haar weights $\varphi_D$ and 
$\widehat{\varphi_D}$ are given by
\begin{align}
&\varphi_D\bigl(\Pi(\mu_k\otimes L_f)\bigr)=\hat{\varphi}(\mu_k)\varphi(L_f)
=\int k(s)f(1)\,ds,  \notag \\
&\widehat{\varphi_D}(L_a\otimes\mu_b)=\varphi(L_a)\hat{\psi}(\mu_b)
=\varphi(L_a)\hat{\varphi}\bigl(\hat{S}(\mu_b)\bigr)
=\int a(1)b(t^{-1})\,dt.
\notag
\end{align}
Meanwhile, remembering the definitions of $\Pi$ and $W_D$, we have:
\begin{align}
&(\Pi(\mu_k\otimes L_f)\otimes1\otimes1)W_D^*(1\otimes1\otimes L_a\otimes\mu_b)
\xi(s,t,s',t')  \notag \\
&=\int\nabla(z)\nabla(t')k(s)f(z)a(z')b(t')\xi(t'^{-1}z^{-1}szt',t'^{-1}z^{-1}t,
z'^{-1}z^{-1}szs',t')\,dzdz'.
\notag
\end{align}
By change of variables (first $z'\mapsto z^{-1}szz'$, and then $z\mapsto zt'^{-1}$),
it becomes:
\begin{align}
\cdots&=\int\nabla(zt'^{-1})k(s)f(zt'^{-1})a(t'z^{-1}szt'^{-1}z')b(t')
\xi(z^{-1}sz,z^{-1}t,z'^{-1}s',t')\,dzdz'     \notag \\
&=\int\nabla(z)F(s,z,s',t')\xi(z^{-1}sz,z^{-1}t,z'^{-1}s',t')\,dzdz'
\notag \\
&=\bigl(\bigl[\Pi\otimes(L\otimes\mu)\bigr](F)\bigr)\xi(s,t,s',t'),
\notag
\end{align}
where $F(s,z;z',t')=\nabla(t'^{-1})k(s)f(zt'^{-1})a(t'z^{-1}szt'^{-1}z')b(t')\in C_c(G\times 
G\times G\times G)$.  Recall equation~\eqref{PiD(G)}.  Therefore,
\begin{align}
&\bigl\langle L_a\otimes\mu_b\,|\,\Pi(\mu_k\otimes L_f)\bigr\rangle
=(\varphi_D\otimes\widehat{\varphi_D})\bigl(\bigl[\Pi\otimes
(L\otimes\mu)\bigr](F)\bigr)    \notag \\
&=\int F(s,1,1,t^{-1})\,dsdt
=\int\nabla(t)k(s)f(t)a(t^{-1}st)b(t^{-1})\,dsdt.
\notag
\end{align}
\end{proof}

By Theorem~\ref{pairingthm}, we know that this is a valid dual pairing map (at the 
level of dense subalgebras) between $\widehat{D(G)}$ and $D(G)$, satisfying (1),(2),(3) 
of Proposition~\ref{pairing}.  In particular, the property~(1) implies that:
$$
\bigl\langle(L_a\otimes\mu_b)(L_{a'}\otimes\mu_{b'})\,|\,\Pi(\mu_k\otimes L_f)
\bigr\rangle=\bigl\langle(L_a\otimes\mu_b)\otimes(L_{a'}\otimes\mu_{b'})\,|\,\Delta_D
(\Pi(\mu_k\otimes L_f))\bigr\rangle,
$$
which relates the comultiplication $\Delta_D$ on $D(G)$ with the product on 
$\widehat{D(G)}$.

Even though we expressed our dual pairing as between certain subalgebras of 
$\widehat{D(G)}$ and $D(G)$, note that the pairing map is in effect being 
considered at the level of functions in $C_c(G\times G)$.  In that sense, we may 
write the pairing map given in Proposition~\ref{pairingD(G)} as:
\begin{equation}\label{(pairingC_c)}
\langle a\otimes b\,|\,k\otimes f\rangle
=\int\nabla(t)a(t^{-1}st)b(t^{-1})k(s)f(t)\,dsdt.
\end{equation}

Let us now consider the deformed comultiplication ${}_{\mathcal R}\Delta$ proposed 
in the previous section, and by using the dual pairing, try to ``deform'' the algebra 
$C^*(G)\otimes C_0(G)$.  Since the dual pairing is valid only at the level of functions, 
we will first work in the subspace $C_c(G\times G)$.  Formally, we wish to deform 
its product given in equation~\eqref{(undeformedprod)} to a new one, so that the 
new product is ``dual'' to ${}_{\mathcal R}\Delta$, as suggested by equation~\eqref
{(deformedprod)}.  In our case, we look for the ``deformed product'' $\times_{\mathcal R}$, 
satisfying (formally) the following:
$$
\bigl\langle[(a\otimes b)\times_{\mathcal R}(a'\otimes b')]\,|\,k\otimes f\bigr\rangle
=\bigl\langle(a\otimes b)\otimes(a'\otimes b')\,|\,{}_{\mathcal R}\Delta(k\otimes f)\bigr\rangle.
$$

To make some sense of this, we first need to regard ${}_{\mathcal R}\Delta(k\otimes f)$ 
as a (generalized) function on $G\times G$.  So consider $k,f\in C_c(G)$, and consider 
$\Pi(\mu_k\otimes L_f)\in D(G)$.  By definition, and by remembering that ${\mathcal R}
=Z_{34}^*\hat{W}_{14}Z_{34}$, we have:
\begin{align}
&{}_{\mathcal R}\Delta\bigl(\Pi(\mu_k\otimes L_f)\bigr)\xi(s,t,s',t')
={\mathcal R}\Delta_D\bigl(\Pi(\mu_k\otimes L_f)\bigr)\xi(s,t,s',t')       \notag \\
&=\nabla(s)W_D^*\bigl(1\otimes1\otimes\Pi(\mu_k\otimes L_f)\bigr)W_D
\xi(s,t,s^{-1}s's,s^{-1}t')       \notag \\
&=\int\nabla(s)\nabla(z)\nabla(z)k(s's)f(z)\xi(z^{-1}sz,z^{-1}t,z^{-1}s^{-1}s'sz,z^{-1}s^{-1}t')\,dz.
\notag
\end{align}
Remembering the definition of $\Pi$ as given in equation~\eqref{PiD(G)}, we write it 
as:
\begin{align}
\cdots&=\int\nabla(z)\nabla(z')F(s,z;s',z')\xi(z^{-1}sz,z^{-1}t;z'^{-1}s'z',z'^{-1}t')\,dzdz'  
\notag \\
&=\bigl[(\Pi\otimes\Pi)(F)\bigr]\xi(s,t;s',t'),
\notag
\end{align}
where $F(s,z;s',z')=\nabla(s)\nabla(z)\nabla(z'^{-1})k(s's)f(z)\delta_{z'}(sz)$.  [Here, $\delta_{z'}(sz)$
is a ``delta function'', such that for any function $g$, we have: $\int g(z')
\delta_{z'}(sz)\,dz'=g(sz)$.]  It is true that $F$ is not really a function in 
$C_c(G\times G\times G\times G)$, but for our purposes, we may regard $F$ as a (generalized) 
``function'' expression corresponding to ${}_{\mathcal R}\Delta\bigl(\Pi(\mu_k\otimes L_f)\bigr)
\in M\bigl(D(G)\otimes D(G)\bigr)$.

Next, use equation~\eqref{(pairingC_c)} to compute the dual pairing (again formally). 
We then have:
\begin{align}
&\bigl\langle(a\otimes b)\otimes(a'\otimes b')\,|\,{}_{\mathcal R}\Delta(k\otimes f)
\bigr\rangle=\bigl\langle(a\otimes b)\otimes(a'\otimes b')\,|\,F\bigr\rangle
\notag \\
&=\int\nabla(t)\nabla(t')a(t^{-1}st)b(t^{-1})a'(t'^{-1}s't')b'(t'^{-1})F(s,t;s',t')\,dsdtds'dt'
\notag \\
&=\int\nabla(t)\nabla(s)\nabla(t)a(t^{-1}st)b(t^{-1})a'(t^{-1}s^{-1}s'st)b'(t^{-1}s^{-1})
k(s's)f(t)\,dsdtds'. 
\notag
\end{align}
By change of variables (letting $s'\mapsto s's^{-1}$ and then letting $s\mapsto tst^{-1}$), 
it becomes:
\begin{align}
\cdots&=\int\nabla(t)a(s)b(t^{-1})a'(s^{-1}t^{-1}s't)b'(s^{-1}t^{-1})k(s')f(t)\,dsds'dt
\notag \\
&=\int\nabla(t)G(t^{-1}s't,t^{-1})k(s')f(t)\,ds'dt
=\langle G\,|\,k\otimes f\rangle,
\notag
\end{align}
where $G(t^{-1}s't,t^{-1})=\int a(s)b(t^{-1})a'(s^{-1}t^{-1}s't)b'(s^{-1}t^{-1})\,ds$.
From which it follows that $G(p,t)=\int a(z)b(t)a'(z^{-1}p)b'(z^{-1}t)\,dz$.

Motivated by these computations (although not fully rigorous and depend on formal 
computations), we propose to define the ``deformed product'' $\times_{\mathcal R}$ 
on $C_c(G\times G)$, as follows:
$$
\bigl[(a\otimes b)\times_{\mathcal R}(a'\otimes b')\bigr](s,t)=G(s,t)
=\int a(z)b(t)a'(z^{-1}s)b'(z^{-1}t)\,dz.
$$
Observe that $\times_{\mathcal R}$ is indeed a valid associative product on 
$C_c(G\times G)$.  See below.

\begin{prop}\label{deformedprod}
On $C_c(G\times G)$, define the ``deformed product'' $\times_{\mathcal R}$, as follows:
$$
\bigl[(a\otimes b)\times_{\mathcal R}(a'\otimes b')\bigr](s,t)
=\int a(z)b(t)a'(z^{-1}s)b'(z^{-1}t)\,dz.
$$
It is a valid associative product on $C_c(G\times G)$, and is ``dual'' to the deformed 
comultiplication ${}_{\mathcal R}\Delta$, in the (formal) sense described above.
\end{prop}

Showing that $\times_{\mathcal R}$ is indeed an associative product on $C_c(G\times G)$
is quite straightforward.  In fact, we can actually construct a $C^*$-algebra that 
contains $\bigl(C_c(G\times G),\times_{\mathcal R}\bigr)$ as a dense subalgebra. 
The method is to follow the standard procedure for constructing a crossed product 
$C^*$-algebra (where $G$ acts on $C_0(G)$ by translation $\tau$).

To be more specific, regard a typical element $a\otimes b\in C_c(G\times G)$ as an 
element $F\in C_c\bigl(G,C_0(G)\bigr)$.  We can then form the space $L^1\bigl(G,C_0(G)
\bigr)$, by completing $C_c\bigl(G,C_0(G)\bigr)$ with respect to the following norm:
$$
\|F\|_1=\int_G\bigl\|F(s)\bigr\|_{\infty}\,ds=\int_G\operatorname{sup}_{t\in G}
\bigl|F(s,t)\bigr|\,ds.
$$
On this $L^1$-space, we can consider the twisted convolution product and the adjoint 
operation, twisted by $\tau$, obtaining the ${}^*$-algebra $L^1\bigl(G,C_0(G)\bigr)$. 
Namely,
\begin{align}
(F\ast G)(s)&=\int_G F(z)\tau_z\bigl(G(z^{-1}s)\bigr)\,dz,
\notag \\
F^*(s)&=\nabla(s^{-1})\tau_s\bigl(F(s^{-1})^*\bigr).
\notag
\end{align}
The crossed product $C^*$-algebra $C_0(G)\rtimes_{\tau}G$ is defined to be the 
enveloping $C^*$-algebra of the ${}^*$-algebra $L^1\bigl(G,C_0(G)\bigr)$.

By viewing $F$ and $G$ as functions on $G\times G$, the multiplication and the 
${}^*$-operation on the $L^1$-algebra become:
\begin{align}
(F\ast G)(s,t)&=\int_G F(z,t)G(z^{-1}s,z^{-1}t)\bigr)\,dz,
\notag \\
F^*(s,t)&=\nabla(s^{-1})\overline{F(s^{-1},s^{-1}t)}.
\notag
\end{align}
Observe that the twisted multiplication is none other than the deformed product 
$\times_{\mathcal R}$ given in Proposition~\ref{deformedprod}.  Therefore, the 
crossed product $C^*$-algebra $B=C_0(G)\rtimes_{\tau}G$ is a $C^*$-algebra 
containing $\bigl(C_c(G\times G),\times_{\mathcal R}\bigr)$ as a dense subalgebra.

\begin{prop}\label{k(h)}
In view of the above discussion, we may regard the $C^*$-algebra $B=C_0(G)
\rtimes_{\tau}G$ as a ``deformed $\widehat{D(G)}$'', whose product is dual to 
the ``deformed comultiplication'' ${}_{\mathcal R}\Delta$ on $D(G)$.  It contains 
$\bigl(C_c(G\times G),\times_{\mathcal R}\bigr)$ as a dense subalgebra.  Meanwhile, 
it is known that there exists an isomorphism of $C^*$-algebras between $C_0(G)
\rtimes_{\tau}G$ (which is sometimes called the ``Weyl algebra'') and the 
$C^*$-algebra  of compact operators ${\mathcal K}\bigl(L^2(G)\bigr)$.  That is,
$$
C_0(G)\rtimes_{\tau}G\cong{\mathcal K}\bigl(L^2(G)\bigr).
$$
\end{prop}

As for the second characterization, see, for instance, \cite{Rfhcr}.  By the way, 
note that in the von Neumann algebraic setting, our result would have been not much 
illuminating, since $\overline{{\mathcal K}({\mathcal H})}^w={\mathcal B}({\mathcal H})$. 
This is the reason why we have chosen to work with the $C^*$-algebra framework in 
Sections~4 and 5.

Compare now with the finite-dimensional case, considered by Lu \cite{Ludouble}, 
\cite{Mo}.  Lu's result says that given a Hopf algebra $H$, the twisting (via the 
$R$-matrix) of the dual of the quantum double turns out to be isomorphic to the 
``smash product'' $H\#\hat{H}$, which is in turn isomorphic to $\operatorname{End}(H)$ 
(see \S9 of \cite{Mo}).  A similar result was obtained in \cite{DeVD}, in the (also 
algebraic) setting of multiplier Hopf algebras.  Our result in Proposition~\ref{k(h)} 
may be viewed as the $C^*$-algebraic counterpart to these results.

\section{Toward the general case.}

Our program of finding a ``twisted $\widehat{A_D}$'' was successful in the ordinary 
group case, mainly because the dual pairing was simple to work with at the level of 
a nice subspace of continuous functions, namely $C_c(G)\subseteq A$.  On the other 
hand, we know that the dual pairing is harder to work with in the general locally 
compact quantum group case.  If we can reduce a little the role being played by the 
actual dual pairing formula, it is likely to lead us to an approach that is more 
general.  

We believe that working with the generalized Fourier transform (as defined earlier) 
could be useful.  In addition, while we wish to keep the overall strategy of the 
previous section, we also wish to find an approach that relies less on the existence 
of a dense subspace consisting of continuous functions.  To find such an approach, 
let us first review the following fact.

Suppose that $(M,\Delta)$ is an arbitrary (von Neumann algebraic) locally compact 
quantum group, with its multiplicative unitary operator $W$.  Recall from Section~2 
that its dual object $\hat{M}$ is given by 
$$
\hat{M}=\overline{\bigl\{(\omega\otimes\operatorname{id})(W):\omega\in M_*\bigr\}}^w.
$$
What this means is that the von Neumann algebra $\hat{M}$ is generated by the pre-dual 
$M_*$ of $M$, via the ``regular representation'' $\lambda:\omega\mapsto(\omega\otimes
\operatorname{id})(W)$.  Moreover, the operator multiplication makes $M_*$ to be 
considered as an algebra.  See Lemma~\ref{M_*} below:

\begin{lem}\label{M_*}
Let $(M,\Delta)$ be a locally compact quantum group, with its multiplicative unitary 
operator $W$.  Denote by $M_*$ the pre-dual of the von Neumann algebra $M$.  Then 
$M_*$ can be given a natural algebra structure, together with a densely defined 
${}^*$-operation:
\begin{enumerate}
  \item For $\omega,\omega'\in M_*$, we have: $\lambda(\omega)\lambda(\omega')
=\lambda(\mu)$ in $\hat{M}$, where $\mu\in M_*$ is such that
$$
\mu(x)=(\omega\otimes\omega')(\Delta x),\qquad {\text {for $x\in M$.}}
$$
  \item Write $\omega\in M_*^{\sharp}$, if $\omega\in M_*$ is such that there 
exists an element $\omega^{\sharp}\in M_*$, given by: 
$$
\omega^{\sharp}(x)
=\bar{\omega}\bigl(S(x)\bigr)=\overline{\omega\bigl([S(x)]^*\bigr)},
\qquad{\text {for all $x\in{\mathcal D}(S)$}}.  
$$
Then we have: $\bigl[\lambda(\omega)\bigr]^*=\lambda(\omega^{\sharp})$ as operators 
in $\hat{M}$.  Meanwhile, the subspace $M_*^{\sharp}$ is a dense subalgebra (in the 
sense of (1) above) of $M_*$, which is closed under taking ${}^{\sharp}$.
\end{enumerate}
\end{lem}

\begin{rem}
A similar result exist with roles of $M$ and $\hat{M}$ reversed.  That is, we may 
think of the von Neumann algebra $M$ being generated by the pre-dual $\hat{M}_*$ 
of $\hat{M}$, via the ``regular representation'' $\hat{\lambda}:\theta\mapsto
(\operatorname{id}\otimes\theta)(\hat{W})=(\theta\otimes\operatorname{id})(W^*)$. 
The situation is basically the same.  All these are immediate consequences of the 
fact that the multiplicative unitary operator $W$ associated with a locally compact 
quantum group is ``manageable''.  See the fundamental papers \cite{BS} and \cite{Wr7}.
\end{rem}

For $M=L^{\infty}(G)$, it is well-known that $M_*=L^1(G)$.  So $M_*$, with its algebra 
structure given by Lemma~\ref{M_*}, is a sort of an $L^1$-algebra that generates the 
von Neumann algebra $\hat{M}$.  This observation suggests that to ``deform'' $\hat{M}$ 
(or $\hat{A}$), we may begin by deforming the algebra structure on $M_*$.  

To follow this strategy in our case, consider now the quantum double $(N_D,\Delta_D)$, 
whose $C^*$-algebraic counterpart is $(A_D,\Delta_D)$.  The multiplicative unitary 
operator is $W_D$, as defined in equation~\eqref{(WDdef)}.  To ``deform'' 
$\widehat{A_D}$, consider the pre-dual $(N_D)_*$ of the von Neumann algebra $N_D$, 
and let us introduce a new multiplication on it, as follows:

\begin{prop}\label{N_D_*}
Let $(N_D)_*$ denote the pre-dual of the von Neumann algebra $N_D$.  For $\omega,
\omega'\in(N_D)_*$, define $\omega\ast_{\mathcal R}\omega'\in(N_D)_*$ by
$$
(\omega\ast_{\mathcal R}\omega')(x):=(\omega\otimes\omega')\bigl({\mathcal R}
\Delta_D(x)\bigr),\qquad {\text {for $x\in M_D$}}.
$$
Then $\ast_{\mathcal R}$ is an associative multiplication on $(N_D)_*$.
\end{prop}
 
\begin{proof}
The associativity of $\ast_{\mathcal R}$ is an immediate consequence of the 
coassociativity of the map $x\mapsto{\mathcal R}\Delta_D(x)={}_{\mathcal R}
\Delta(x)$, as noted in Proposition~\ref{RDeltacoassociativity}.
\end{proof}

Let us now look for a representation $Q$ of $\bigl((N_D)_*,\ast_{\mathcal R}
\bigr)$ into ${\mathcal B}({\mathcal H}\otimes{\mathcal H})$.  First, recall that 
the operators $(\omega\otimes\operatorname{id})(W_D)$, $\omega\in(N_D)_*$, 
are dense in $\widehat{A_D}$.  If we denote by $\hat{\Lambda}_{D}$ the GNS 
map for the Haar weight $\widehat{\varphi_D}$ of $\widehat{A_D}$, we thus 
know that the elements of the form $\hat{\Lambda}_D\bigl((\omega\otimes
\operatorname{id})(W_D)\bigr)$ are dense in ${\mathcal H}\otimes{\mathcal H}$. 
This suggests the following definition of the ``representation'' $Q$.  At the moment, 
no compatible ${}^*$-structure is specified on $\bigl((N_D)_*,\ast_{\mathcal R}\bigr)$, 
so we only know that $Q$ is an algebra homomorphism.

\begin{defn}\label{Qrep}
Define $Q:\bigl((N_D)_*,\ast_{\mathcal R}\bigr)\to{\mathcal B}({\mathcal H}\otimes
{\mathcal H})$ by
$$
Q(\omega)\hat{\Lambda}_D\bigl((\nu\otimes\operatorname{id})(W_D)\bigr):=
\hat{\Lambda}_D\bigl(([\omega\ast_{\mathcal R}\nu]\otimes\operatorname{id})
(W_D)\bigr).
$$
Since $\ast_{\mathcal R}$ is associative, and since the $\hat{\Lambda}_D\bigl((\nu
\otimes\operatorname{id})(W_D))\bigr)$, $\nu\in(N_D)_*$, are dense in the 
Hilbert space ${\mathcal H}\otimes{\mathcal H}$, this is certainly an algebra 
homomorphism, preserving the multiplication.  Namely, $Q(\omega)Q(\omega')
=Q(\omega\ast_{\mathcal R}\omega')$.
Define $B$ as the $C^*$-subalgebra of ${\mathcal B}({\mathcal H}\otimes{\mathcal H})$ 
generated by the $Q(\omega)$, $\omega\in(N_D)_*$.  Then $B$ may be considered as 
the ``deformed $\widehat{A_D}$''.
\end{defn}

Unfortunately, finding a more concrete description of the $C^*$-algebra $B$ seems 
rather difficult.  Recall that even before deforming, the $C^*$-algebra $\widehat
{A_D}$ itself could be rather complicated in general.  See comments following 
Proposition~\ref{dualqdouble} and also see \S9 of \cite{BjVa}.  It is likely that 
the $C^*$-algebra $B$ may be just as complicated.

In view of this obstacle, while we will try to push our strategy in the general case, 
we will soon restrict our attention to the case of $D(G)$, and re-formulate the 
result of Section~5 using the new approach suggested by Definition~\ref{Qrep}. 
We hope that this alternative perspective can shed some light on the general case 
in the future.

With these remarks in mind, let us learn a little more about the subalgebra 
$Q\bigl((N_D)_*\bigr)\subseteq{\mathcal B}({\mathcal H}\otimes{\mathcal H})$. 
Suppose $\omega,\nu\in(N_D)_*$, and let $x\in{\mathcal N}_D$ be arbitrary. 
Then by Definition~\ref{Qrep}, we have:
\begin{align}
&\bigl\langle Q(\omega)\hat{\Lambda}_D\bigl((\nu\otimes\operatorname{id})(W_D)\bigr),
\Lambda_D(x)\bigr\rangle=\bigl\langle\hat{\Lambda}_D\bigl(([\omega\ast_{\mathcal R}\nu]
\otimes\operatorname{id})(W_D)\bigr),\Lambda_D(x)\bigr\rangle  \notag \\
&=(\omega\ast_{\mathcal R}\nu)(x^*)  
=(\omega\otimes\nu)\bigl({\mathcal R}\Delta_D(x^*)\bigr)
=(\omega\otimes\nu)(\Delta_D^{\operatorname{cop}}(x^*){\mathcal R}\bigr)
\notag \\
&=\bigl\langle\hat{\Lambda}_D\bigl((\omega\otimes\operatorname{id})(W_D)\bigr)
\otimes\hat{\Lambda}_D\bigl((\nu\otimes\operatorname{id})(W_D)\bigr),(\Lambda_D
\otimes\Lambda_D)\bigl({\mathcal R}^*\Delta_D^{\operatorname{cop}}(x)\bigr)
\bigr\rangle.
\notag
\end{align}
Here $\langle\ ,\ \rangle$ denotes the inner product, the first two are on ${{\mathcal H}
\otimes{\mathcal H}}$, while the last one is on ${({\mathcal H}\otimes{\mathcal H})
\otimes({\mathcal H}\otimes{\mathcal H})}$.  The second and the fifth equalities 
are just using the definition of $\hat{\Lambda}_D$, as in equation~\eqref{(dualHaar)}. 
The third equality is from Proposition~\ref{N_D_*}, and the fourth equality is the result 
of Proposition~\ref{rmatrix}\,(3).

Meanwhile, we know from Section~3 that $\widehat{N_D}=N\otimes\hat{N}$, which 
means that the elements $(\omega\otimes\operatorname{id})(W_D)$, $\omega
\in(N_D)_*$, are approximated by the elements of the form, $a\otimes b$, where 
$a\in{\mathcal A}\,(\subseteq N)$, $b\in\hat{\mathcal A}\,(\subseteq\hat{N})$. 
Therefore, the product $\ast_{\mathcal R}$ from Proposition~\ref{N_D_*} determines 
the ``deformed product'', $\times_{\mathcal R}$, on a certain dense subspace of 
$N\otimes\hat{N}$.  Then the computation above may be re-written as follows:
\begin{align}\label{(Rprod1)}
&\bigl\langle\hat{\Lambda}_D\bigl((a\otimes b)\times_{\mathcal R}(a'\otimes b')\bigr),
\Lambda_D(x)\bigr\rangle       \notag \\
&=\bigl\langle\hat{\Lambda}_D(a\otimes b)\otimes\hat{\Lambda}_D(a'\otimes b'),
{\mathcal R}^*(\Lambda_D\otimes\Lambda_D)\bigl(\Delta_D^{\operatorname{cop}}(x)
\bigr)\bigr\rangle       \notag \\
&=\bigl\langle{\mathcal R}[\hat{\Lambda}_D(a\otimes b)\otimes\hat{\Lambda}_D
(a'\otimes b')],(\Lambda_D\otimes\Lambda_D)\bigl(\Delta_D^{\operatorname{cop}}(x)
\bigr)\bigr\rangle.
\end{align}
Here we are using the fact $(\Lambda_D\otimes\Lambda_D)\bigl({\mathcal R}^*
\Delta_D^{\operatorname{cop}}(x)\bigr)={\mathcal R}^*\bigl[(\Lambda_D\otimes\Lambda_D)
\bigl(\Delta_D^{\operatorname{cop}}(x)\bigr)\bigr]$, which is true since ${\mathcal R}
\in N_D\otimes N_D$ and since the GNS representation associated with $\Lambda_D$ 
is just the inclusion map $N_D\subseteq{\mathcal B}({\mathcal H}\otimes{\mathcal H})$.

Let us denote by ${\mathcal F}_D$ and ${\mathcal F}^{-1}_D$ the Fourier transform 
and the inverse Fourier transform between certain dense subalgebras of $N_D$ and 
$\widehat{N_D}$, defined in the same way as in Theorem~\ref{pairingthm}.  By the 
property of the Fourier transform (see Propositions~3.5 and 3.7 of \cite{BJKqfourier}),
it is known that $\hat{\Lambda}_D\bigl({\mathcal F}_D(x)\bigr)=\Lambda_D(x)$ and 
that $\Lambda_D\bigl({\mathcal F}^{-1}_D(y)\bigr)=\hat{\Lambda}_D(y)$, where 
$x\in N_D$ and $y\in\widehat{N_D}$ are assumed to be contained in suitable 
domains.  We thus have:
\begin{align}\label{(Rprod2)}
&{\mathcal R}[\hat{\Lambda}_D(a\otimes b)\otimes\hat{\Lambda}_D(a'\otimes b')]
={\mathcal R}\bigl[{\Lambda}_D\bigl({\mathcal F}^{-1}_D(a\otimes b)\bigr)\otimes
{\Lambda}_D\bigl({\mathcal F}^{-1}_D(a'\otimes b')\bigr)\bigr]    \notag \\
&=(\Lambda_D\otimes\Lambda_D)\bigl({\mathcal R}\bigl[{\mathcal F}^{-1}_D(a\otimes b)
\otimes{\mathcal F}^{-1}_D(a'\otimes b')\bigr]\bigr)   \notag \\
&=(\hat{\Lambda}_D\otimes\hat{\Lambda}_D)\bigl[({\mathcal F}_D\otimes{\mathcal F}_D)
\bigl({\mathcal R}\bigl[{\mathcal F}^{-1}_D(a\otimes b)\otimes{\mathcal F}^{-1}_D
(a'\otimes b')\bigr]\bigr)\bigr]. 
\end{align}

\begin{rem}
If we formally extend the Fourier transform, then by the Fourier inversion theorem, 
we may write ${\mathcal R}=({\mathcal F}^{-1}_D\otimes{\mathcal F}^{-1}_D)
\bigl(({\mathcal F}_D\otimes{\mathcal F}_D)({\mathcal R})\bigr)$.  Then the expression 
in the last line above is essentially the ``convolution product'', as defined in 
Proposition~3.11 of \cite{BJKqfourier}.  That is, 
$$
({\mathcal F}_D\otimes{\mathcal F}_D)\bigl({\mathcal R}\bigl[{\mathcal F}^{-1}_D(a\otimes b)
\otimes{\mathcal F}^{-1}_D(a'\otimes b')\bigr]\bigr)=({\mathcal F}_D\otimes{\mathcal F}_D)
({\mathcal R})\ast\bigl((a\otimes b)\otimes(a'\otimes b')\bigr).
$$
We may use the result in \cite{BJKqfourier} to write down an alternative description for 
the convolution product, using the Haar weight and the antipode map.
\end{rem}

Comparing our computations in this section with Proposition~\ref{pairing} (1), 
we see that 
\begin{align}\label{(Rproduct)}
&(a\otimes b)\times_{\mathcal R}(a'\otimes b')   \notag \\
&=\bigl((m_{N})_{31}\otimes(m_{\hat{N}})_{42}\bigr)\bigl[({\mathcal F}_D\otimes
{\mathcal F}_D)\bigl({\mathcal R}\bigl[{\mathcal F}^{-1}_D(a\otimes b)\otimes
{\mathcal F}^{-1}_D(a'\otimes b')\bigr]\bigr)\bigr],
\end{align}
where $m_{N}$ and $m_{\hat{N}}$ denote the multiplications on $N$ and $\hat{N}$, 
respectively.  

While the formula given in equation~\eqref{(Rproduct)} is not entirely rigorous, 
it does give us a workable description (assuming the details like the operator $W_D$, 
the Haar weights, and the Fourier transforms are known) of the ``deformed product'' 
$\times_{\mathcal R}$, on a dense subspace contained in ${\mathcal A}\otimes
\hat{\mathcal A}$.  This is essentially the multiplication on $(N_D)_*$ given 
in Proposition~\ref{N_D_*}.  

As we indicated earlier in the section, we do not plan to carry out the computations 
in full generality, which seems rather difficult.  Instead, let us from now on return to 
the set up and the notations given in Section~5, corresponding to $N={\mathcal L}(G)$ 
and $\hat{N}=L^{\infty}(G)$.  As before, it is convenient to work with the space of 
functions having compact support.

\begin{lem}\label{FTN_D}
Let $a,b\in C_c(G)$ and consider $L_a\otimes\mu_b\in N\otimes\hat{N}=\widehat{N_D}$. 
Then:
$$
{\mathcal F}^{-1}_D(L_a\otimes\mu_b)=(\mu_a\otimes1)Z^*(1\otimes L_{\tilde{b}})Z
=\Pi(\mu_a\otimes L_{\tilde{b}})\in N_D,
$$
where $\tilde{b}(t)=\nabla(t^{-1})b(t)$. [Recall that $\nabla$ is the modular function.]
\end{lem}

\begin{proof}
By definition, 
$$
{\mathcal F}^{-1}_D(L_a\otimes\mu_b)=(\operatorname{id}\otimes\widehat{{\varphi}_D})
\bigl(W_D^*([1\otimes1]\otimes[L_a\otimes\mu_b])\bigr).
$$
Since $\widehat{\varphi_D}=\varphi\otimes\hat{\psi}$ (see Proposition~\ref{haarweight} (2)),
this becomes:
\begin{align}
{\mathcal F}^{-1}_D(L_a\otimes\mu_b)&=(\operatorname{id}\otimes\operatorname{id}
\otimes\varphi\otimes\hat{\psi})\bigl(\hat{W}^*_{13}Z^*_{12}W^*_{24}Z_{12}(1\otimes1
\otimes L_a\otimes\mu_b)\bigr)   \notag \\
&=(\operatorname{id}\otimes\operatorname{id}\otimes\varphi\otimes\hat{\psi})
\bigl([\hat{W}^*(1\otimes L_a)]_{13}Z^*_{12}[W^*(1\otimes\mu_b)]_{24}Z_{12}\bigr)
\notag \\
&=\bigl(\bigl[(\operatorname{id}\otimes\varphi)\bigl(\hat{W}^*(1\otimes L_a)
\bigr)\bigr]\otimes1\bigr)Z^*\bigl(1\otimes\bigl[(\operatorname{id}\otimes\hat{\psi})
\bigl(W^*(1\otimes\mu_b)\bigr)\bigr]\bigr)Z.
\notag
\end{align}
But remembering that $\hat{W}=\Sigma W^*\Sigma$, we have:
$$
(\operatorname{id}\otimes\varphi)\bigl(\hat{W}^*(1\otimes L_a)\bigr)
=(\varphi\otimes\operatorname{id})\bigl(W(L_a\otimes1)\bigr)={\mathcal F}(L_a)=\mu_a,
$$
where the last result was shown in Section~5 of \cite{BJKqfourier}, and can be obtained 
by a direct computation.  Similarly, $(\operatorname{id}\otimes\hat{\varphi})
\bigl(W^*(1\otimes\mu_b)\bigr)={\mathcal F}^{-1}(\mu_b)=L_b$.  Since $\hat{\psi}$ 
and $\hat{\varphi}$ are related by the modular function (in general, related via 
the ``modular operator''),  we can show without much difficulty that
$$
(\operatorname{id}\otimes\hat{\psi})\bigl(W^*(1\otimes\mu_b)\bigr)=L_{\tilde{b}},
$$
where $\tilde{b}\in C_c(G)$ is as defined above.

Combining the results, we indeed have:
$$
{\mathcal F}^{-1}_D(L_a\otimes\mu_b)=(\mu_a\otimes1)Z^*(1\otimes L_{\tilde{b}})Z
=\Pi(\mu_a\otimes L_{\tilde{b}}).
$$
\end{proof}

While the above Lemma was formulated for the case of $N={\mathcal L}(G)$ and $\hat{N}
=L^{\infty}(G)$, we can see from the proof that a reasonable generalization (using the 
Fourier transform) could be given for more general settings.  In this paper, we will be 
content with the current description, since we will be using a computational method 
in what follows.

Let us now put together the results so far.  In our case, with the Fourier transform being 
rather simple (see Lemma~\ref{FTN_D}), the actual computation is not too difficult.  By 
a straightforward computation, the expression in equation~\eqref{(Rprod2)} becomes: 
$$
({\mathcal F}_D\otimes{\mathcal F}_D)\bigl({\mathcal R}\bigl[{\mathcal F}^{-1}_D
(L_a\otimes\mu_b)\otimes{\mathcal F}^{-1}_D(L_{a'}\otimes\mu_{b'})\bigr]\bigr)
=(L\otimes\mu\otimes L\otimes\mu)(F),   \notag
$$
where $F\in C_c(G\times G\times G\times G)$ is given by
$$
F(s,t,s',t')=\nabla(s)a(s)b(t)a'(s^{-1}s's)b'(s^{-1}t').
$$
Next, equation~\eqref{(Rproduct)} will provide us with the deformed product 
$\times_{\mathcal R}$ on $C_c(G\times G)$, as follows:
\begin{align}\label{((defprod))}
&[(a\times b)\times_{\mathcal R}(a'\otimes b')](s,t)
=\bigl[\bigl((m_N)_{31}\otimes(m_{\hat{N}})_{42}\bigr)(F)\bigr](s,t)
\notag \\
&=\int F(z^{-1}s,t,z,t)\,dz
=\int\nabla(z^{-1}s)a(z^{-1}s)b(t)a'(s^{-1}zzz^{-1}s)b'(s^{-1}zt)\,dz
\notag \\
&=\int\nabla(s)a(zs)b(t)a'(s^{-1}z^{-1}s)b'(s^{-1}z^{-1}t)\,dz
\notag \\
&=\int a(z)b(t)a'(z^{-1}s)b'(z^{-1}t)\,dz.
\end{align}
In the fourth and fifth equalities, we used the change of variables, $z\mapsto z^{-1}$, 
and then $z\mapsto zs^{-1}$.

Observe that we obtain the multiplication on $C_c(G\times G)$ that is exactly the same 
as the one given in Proposition~\ref{deformedprod}.  As we indicated earlier, this is 
none other than the deformed product on $(N_D)_*$ as in Proposition~\ref{N_D_*}. 
Moreover, the $C^*$-algebra $B=C_0(G)\rtimes_{\tau}G$, which was shown in 
Section~5 to be the completion of $\bigl(C_c(G\times G),\times_{\mathcal R}\bigr)$ 
will be the $C^*$-algebra generated by the $Q(\omega)$, $\omega\in(N_D)_*$, 
as described in Definition~\ref{Qrep}.

The computations here support our definition of the ``deformed $\widehat{A_D}$'' 
as given in Definition~\ref{Qrep}.  It is an improvement, since the definition is given 
in a fairly general manner, and since a very straightforward way of construction is 
also obtained via equation~\eqref{(Rproduct)}.

However, we note that the last part of the process, realizing the product given in 
equation~\eqref{(Rproduct)}, needs further improvement.  While the method is 
reasonably practical in the sense that once we have enough information (about 
the Haar weight, the multiplicative unitary operator, and the Fourier transform) 
we can carry out the construction, it will be more desirable if we can reduce 
our dependence on specific computational results.

With this remark in mind, let us include the following observation, which may be 
relevant for future generalization of our program:

\begin{prop}
Let the notations be as above.  Then:
\begin{align}
B=C_0(G)\rtimes_{\tau}G&=\overline{\bigl\{(1\otimes\mu^{\operatorname{op}}_b)
\bigl(\Delta(L_a)\bigr):a,b\in C_c(G)\bigr\}}^{\|\ \|}
\notag \\
&=(1\otimes\hat{A}^{\operatorname{op}})\Delta(A)\,\subseteq{\mathcal B}
({\mathcal H}\otimes{\mathcal H}).
\notag
\end{align}
\end{prop}

\begin{rem}
Here, $\hat{A}^{\operatorname{op}}$ is the $C^*$-algebra corresponding to $\hat{N}'$, 
equipped with the opposite multiplication, being denoted by $\mu^{\operatorname{op}}$. 
In our case, working with $\mu^{\operatorname{op}}$ is just nominal, since the product 
on $\hat{N}=L^{\infty}(G)$ is already known to be commutative.  We nevertheless chose 
to use $\mu^{\operatorname{op}}$, anticipating a possible future generalization. 
Indeed, the description above was obtained from some heuristic computations exploiting 
the close relationship between the multiplicative unitary operator $\hat{W}$ and the 
operator ${\mathcal R}=Z^*_{34}\hat{W}_{14}Z_{34}$.
\end{rem}

\begin{proof}
Let $a,b\in C_c(G)$ and let $\xi\in{\mathcal H}\otimes{\mathcal H}$.  Then by 
the results obtained in Section~5, we have:
$$
(1\otimes\mu^{\operatorname{op}}_b)\bigl(\Delta(L_a)\bigr)\xi(s,t)=\int b(t)
a(z)\xi(z^{-1}s,z^{-1}t)\,dz.
$$
Comparing this with the concrete realization we obtained in equation~\eqref{((defprod))} 
for the product on the $C^*$-algebra $B$ (see also Section~5), the result of the 
proposition follows.
\end{proof}

Unless the quantum group $(A,\Delta)$ is ``regular'' (in the sense of Baaj 
and Skandalis \cite{BS}, \cite{Wr7}), the $C^*$-algebra $(1\otimes\hat{A})
\Delta^{\operatorname{cop}}(A)$ is not necessarily isomorphic to ${\mathcal K}
({\mathcal H})$ and in general may be quite complicated (It may not even be 
``type I''. See \cite{VDoamp} and Section~9 of \cite{BjVa}.).  Meanwhile, 
even though we cannot provide a general proof here, several computations 
at the heuristic level (using different examples) seem to suggest that this is 
the correct description for the $C^*$-algebra $B$.  We hope to report on this 
matter in the near future.

\bigskip\bigskip

\providecommand{\bysame}{\leavevmode\hbox to3em{\hrulefill}\thinspace}
\providecommand{\MR}{\relax\ifhmode\unskip\space\fi MR }
\providecommand{\MRhref}[2]{%
  \href{http://www.ams.org/mathscinet-getitem?mr=#1}{#2}
}
\providecommand{\href}[2]{#2}


\begin{thebibliography}{10}

\bibitem{Ab}
E.~Abe, \emph{{H}opf {A}lgebras}, Cambridge Tracts in Mathematics, no.~74,
  Cambridge University Press, 1980.

\bibitem{BS}
S.~Baaj and G.~Skandalis, \emph{Unitaires multiplicatifs et dualit\'e pour les
  produits crois\'es de {$C^*$}-alg\`ebres}, Ann. Scient. \'Ec. Norm. Sup.,
  $4^e$ s\'erie \textbf{t. 26} (1993), 425--488 (French).

\bibitem{BjVa}
S.~Baaj and S.~Vaes, \emph{Double crossed products of locally compact quantum
  groups}, J. Inst. Math. Jussieu \textbf{4} (2005), no.~1, 135--173.

\bibitem{CP}
V.~Chari and A.~Pressley, \emph{A {G}uide to {Q}uantum {G}roups}, Cambridge
  Univ. Press, 1994.

\bibitem{DeVD}
L.~Delvaux and A.~Van Daele, \emph{The {D}rinfeld double versus the
  {H}eisenberg double for an algebraic quantum group}, J. Pure \& Appl. Alg.
  \textbf{190} (2004), 59--84.

\bibitem{Dr}
V.~G. Drinfeld, \emph{Quantum groups}, Proceedings of the International
  Congress of Mathematicians (Berkeley) (A.~M. Gleason, ed.), American
  Mathematical Society, Providence, RI, 1986, pp.~798--820.

\bibitem{BJKqdouble}
B.~J. Kahng, \emph{Quantum double construction in the {$C^*$}-algebra setting
  of certain {H}eisenberg-type quantum groups}, Houston J. Math. \textbf{32}
  (2006), no.~4, 1153--1189.

\bibitem{BJKqfourier}
\bysame, \emph{Fourier transform on locally compact quantum groups}, 2007,
  preprint (accepted to appear in J. Operator Theory, available as
  arXiv:0708.3055, at http://lanl.arXiv.org).

\bibitem{KuVD}
J.~Kustermans and A.~Van Daele, \emph{{$C^*$}-algebraic quantum groups arising
  from algebraic quantum groups}, Int. J. Math. \textbf{8} (1997), no.~8,
  1067--1139.

\bibitem{KuVa}
J.~Kustermans and S.~Vaes, \emph{Locally compact quantum groups}, Ann. Scient.
  \'Ec. Norm. Sup., $4^e$ s\'erie \textbf{t. 33} (2000), 837--934.

\bibitem{KuVavN}
\bysame, \emph{Locally compact quantum groups in the von {N}eumann algebraic
  setting}, Math. Scand. \textbf{92} (2003), no.~1, 68--92.

\bibitem{Ludouble}
J.~H. Lu, \emph{On the {D}rinfeld double and the {H}eisenberg double of a
  {H}opf algebra}, Duke Math. J. \textbf{74} (1994), no.~3, 763--776.

\bibitem{MNW}
T.~Masuda, Y.~Nakagami, and S.~Woronowicz, \emph{A {$C^*$}-algebraic framework
  for quantum groups}, Internat. J. Math. \textbf{14} (2003), no.~9, 903--1001.

\bibitem{Mo}
S.~Montgomery, \emph{Hopf {A}lgebras and {T}heir {A}ctions on {R}ings}, CBMS
  Regional Conference Series in Mathematics, no.~82, American Mathematical
  Society, 1993.

\bibitem{PW}
P.~Podles and S.~L. Woronowicz, \emph{Quantum deformation of {L}orentz group},
  Comm. Math. Phys. \textbf{130} (1990), 381--431.

\bibitem{Rfhcr}
M.~A. Rieffel, \emph{On the uniqueness of the {H}eisenberg commutation
  relations}, Duke Math J. \textbf{39} (1972), 745--752.

\bibitem{Tk2}
M.~Takesaki, \emph{{T}heory of {O}perator {A}lgebras {II}}, Encyclopaedia of
  Mathematical Sciences, no. 125, Springer-Verlag, 2002.

\bibitem{VDoamp}
S.~Vaes and A.~{Van Daele}, \emph{The {H}eisenberg commutation relations,
  commuting squares and the {H}aar measure on locally compact quantum groups},
  Proceedings of the OAMP conference (Constan\c ta, 2001), Theta, Bucharest,
  2003, pp.~379--400.

\bibitem{VDmult}
A.~{Van Daele}, \emph{Multiplier {H}opf algebras}, Trans. Amer. Math. Soc.
  \textbf{342} (1994), 917--932.

\bibitem{VDvN}
\bysame, \emph{Locally compact quantum groups. {A} von {N}eumann algebra
  approach}, 2006, preprint (available as math.OA/0602212 at
  http://lanl.arXiv.org).

\bibitem{VDFourier}
\bysame, \emph{The {F}ourier transform in quantum group theory}, 2007, preprint
  (available as math.RA/0609502 at http://lanl.arXiv.org).

\bibitem{Wr7}
S.~L. Woronowicz, \emph{From multiplicative unitaries to quantum groups},
  Internat. J. Math. \textbf{7} (1996), no.~1, 127--149.

\bibitem{Ya}
T.~Yamanouchi, \emph{Double group construction of quantum groups in the von
  {N}eumann algebra framework}, J. Math. Soc. Japan \textbf{52} (2000), no.~4,
  807--834.

\end{thebibliography}
\end{document}